\documentclass{article}

\usepackage[margin=2cm]{geometry}

\usepackage{amsmath}
\usepackage{amsfonts}
\usepackage{amsthm}
\usepackage{mathrsfs}
\usepackage{amssymb}
\usepackage{tikz}
\usepackage{comment}
\usetikzlibrary{calc}
\usepackage{tikz-cd}
\usepackage{graphicx}
\usepackage{mathtools}
\usepackage{setspace}
\usepackage{blkarray}
\usepackage[colorlinks=true, allcolors=red]{hyperref}
\usepackage[capitalise, 
            noabbrev, 
            ]{cleveref}
\usepackage{graphicx,float,subfigure}
\usepackage{authblk}

\DeclareMathOperator{\Read}{\mathsf{Read}}
\DeclareMathOperator{\adj}{\mathsf{adj}}
\DeclareMathOperator{\sgn}{\mathsf{sgn}}
\DeclareMathOperator{\wt}{\mathsf{wt}}
\DeclareMathOperator{\area}{\mathsf{area}}
\DeclareMathOperator{\level}{\mathsf{level}}
\DeclareMathOperator{\tp}{\mathsf{tp}}

\DeclareMathOperator{\dinv}{\mathsf{dinv}}
\DeclareMathOperator{\diags}{\mathsf{diags}}

\DeclareMathOperator{\LNDP}{LNDP}
\DeclareMathOperator{\SLNDP}{SLNDP}
\DeclareMathOperator{\AD}{AD}
\DeclareMathOperator{\ASM}{ASM}
\DeclareMathOperator{\T}{\mathcal{T}}

\newcommand{\Q}{\mathbb Q}

\newtheorem{theorem}{Theorem}[section]
\newtheorem{lemma}[theorem]{Lemma}
\newtheorem{corollary}[theorem]{Corollary}

\newtheorem{proposition}[theorem]{Proposition}
\newtheorem{conjecture}[theorem]{Conjecture}

\theoremstyle{definition}

\newtheorem{example}[theorem]{Example}
\newtheorem{remark}[theorem]{Remark}

\begin{document}

\title{Domino Tilings, Domino Shuffling, and the Nabla Operator}

\author[1]{Ian Cavey \thanks{\href{mailto:cavey@illinois.edu}{cavey@illinois.edu}}}
\author[2]{Yi-Lin Lee \thanks{\href{mailto:yillee@iu.edu}{yillee@iu.edu}}}
\affil[1]{Department of Mathematics, University of Illinois Urbana-Champaign, Urbana, IL, USA}
\affil[2]{Department of Mathematics, Indiana University, Bloomington, IN, USA}

\date{}
\maketitle
\setlength{\parskip}{1mm}

\begin{abstract}
We study domino tilings of certain regions $R_\lambda$, indexed by partitions $\lambda$, weighted according to generalized area and dinv statistics. These statistics arise from the $q,t$-Catalan combinatorics and Macdonald polynomials. We present a formula for the generating polynomial of these domino tilings in terms of the Bergeron--Garsia nabla operator. When $\lambda = (n^n)$ is a square shape, domino tilings of $R_\lambda$ are equivalent to those of the Aztec diamond of order $n$. In this case, we give a new product formula for the resulting polynomials by domino shuffling and its connection with alternating sign matrices. In particular, we obtain a combinatorial proof of the joint symmetry of the generalized area and dinv statistics.
\end{abstract}

\begin{small}

\vspace{-1mm}
\end{small}


\section{Introduction}\label{sec:introduction}

The study of Macdonald polynomials has produced many interesting combinatorial objects, often expressed as weighted sums over sets of lattice paths. Perhaps the most famous and well-studied such objects are the $q,t$-Catalan numbers introduced by Garsia and Haiman \cite{GH96}, which can be defined combinatorially as the sum over Dyck paths weighted by the area and dinv statistics (see the book \cite{H08} and the references therein). Among the many generalizations of the $q,t$-Catalan numbers are the extension to Schr\"oder paths defined by Egge, Haglund, Killpatrick, and Kremer \cite{EHKK03}, and to nested families of Dyck paths due to Loehr and Warrington \cite{LW08}. All of these objects have natural interpretations in terms of Macdonald polynomials often expressed via the \emph{nabla operator} $\nabla$ on symmetric functions introduced by Bergeron and Garsia \cite{BG99}. In this paper, we study the common generalization to nested families of Schr\"oder paths and their connection to domino tilings.

Let $\Lambda$ denote the ring of symmetric functions in the countably infinite set of variables $\{x_i\}_{i \geq 1}$ with coefficients in the field of rational functions $\Q(q,t).$ We will make use of the Hall inner product $\langle \cdot,\cdot \rangle$ on $\Lambda$, and the following vector space bases for $\Lambda$, each indexed by partitions $\mu$: the monomial symmetric functions $m_\mu$, the elementary symmetric functions $e_\mu$, the complete homogeneous symmetric functions $h_\mu$, the Schur functions $s_\mu$, and the modified Macdonald Polynomials $\tilde{H}_\mu$. For details on the classical bases and the Hall inner product, see the book \cite{Mac79}. The modified Macdonald polynomials $\tilde{H}_\mu$ are variants of the polynomials studied by Macdonald \cite{Mac88} appearing in Garsia and Haiman's work on diagonal harmonics \cite{GH96}. See also \cite[Chapter~2]{H08} for more details.

The \textit{nabla operator} is the unique $\Q(q,t)$-linear map $\nabla$ on $\Lambda$ such that 
\begin{equation}\label{eq.nabladef}
    \nabla(\tilde{H}_\mu)=q^{n(\mu')}t^{n(\mu)}\tilde{H}_\mu,
\end{equation}
for all $\mu$, where $\mu'$ is the conjugate of $\mu$ and $n(\mu) =\sum (i-1)\mu_i$. One can obtain interesting combinatorial objects by applying $\nabla$ to elements of some basis of $\Lambda$ and expanding the result in terms of another bases or, equivalently, taking the inner product with an element of the dual basis. We list some of these combinatorial objects in Table \ref{tab.nabla}. For other formulas of this kind, see \cite{LW08}.
\begin{table}[htb!]
\centering
\begin{tabular}{c|c|c}
  \hline
  Algebraic Object & Combinatorial Model & Special Case of Theorem \ref{thm:MacdonaldIntro} \\
  \hline
  $\langle \nabla(e_n),e_n\rangle$ & the $q,t$-Catalan numbers \cite{GH96}\cite{GH02} & $\lambda = (1^n)$ and $d=0$\\
  \hline
  $\langle \nabla(e_n),h_d e_{n-d}\rangle$ & the $q,t$--Schr\"oder numbers \cite{EHKK03}\cite{Hag04} & $\lambda = (1^n)$ \\
  \hline
  $\sgn(\lambda)\langle \nabla(s_{\lambda}),e_{n}\rangle$ & the $\lambda$-families of Dyck paths \cite{LW08} & $d=0$\\
  \hline
\end{tabular}
\caption{Some combinatorial objects 
 of the nabla operator.}\label{tab.nabla}
\end{table}

Let $\lambda=(\lambda_1,\dots,\lambda_\ell)$ be a partition of $n$ where $n=\lambda_1+\cdots+\lambda_\ell$ and $\lambda_1\geq \lambda_2\geq \cdots\geq \lambda_\ell>0$. We identify $\lambda$ with its Ferrers diagram (in English notation) as in Figure~\ref{fig:lambda}. Loehr and Warrington \cite{LW08} consider strictly nested collections of Dyck paths with starting points $(0,0),(1,1),(2,2)\dots$, and ending points determined by $\lambda$ in the following way. Decompose the Ferrers diagram of $\lambda$ into border strips and label each column with the length of the strip whose rightmost box is in that column (and with label $0$ if there is no such strip) as shown in Figure \ref{fig:lambda} for the partition $\lambda = (4,4,3,3,3,1)$. Let $n_0,\dots,n_k$ denote these labels read from right to left where $k=\lambda_1-1$. A \emph{$\lambda$-family of Dyck paths} is a $(k+1)$-tuple $\pi = (\pi_0,\dots,\pi_k)$ of pairwise non-intersecting lattice paths, where each $\pi_j$ is a lattice path consisting of unit north $(0,1)$ and east $(1,0)$ steps that lies weakly above the diagonal $y=x$, starts at $(j,j)$, and ends at $(j+n_j,j+n_j)$.
\begin{figure}[h]
    \centering
    \subfigure[]{\label{fig:lambda}\includegraphics[height=0.2\textwidth]{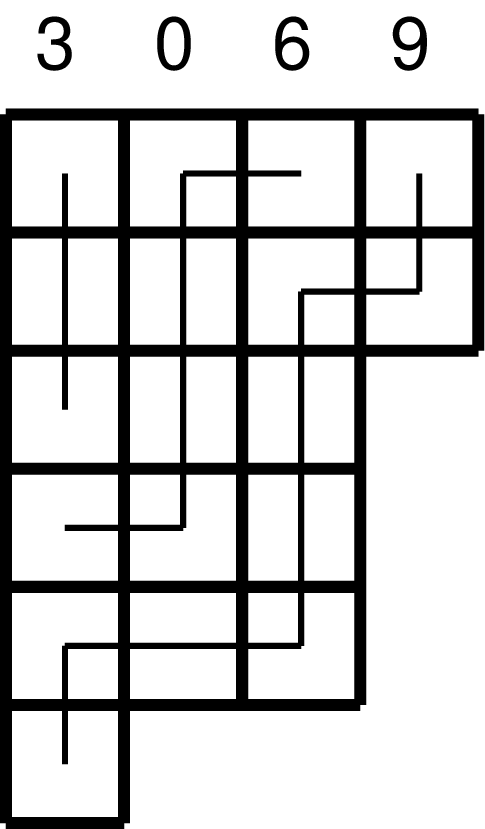}}
    \hspace{10mm}
    \subfigure[]{\label{fig:paths}\includegraphics[height=0.2\textwidth]{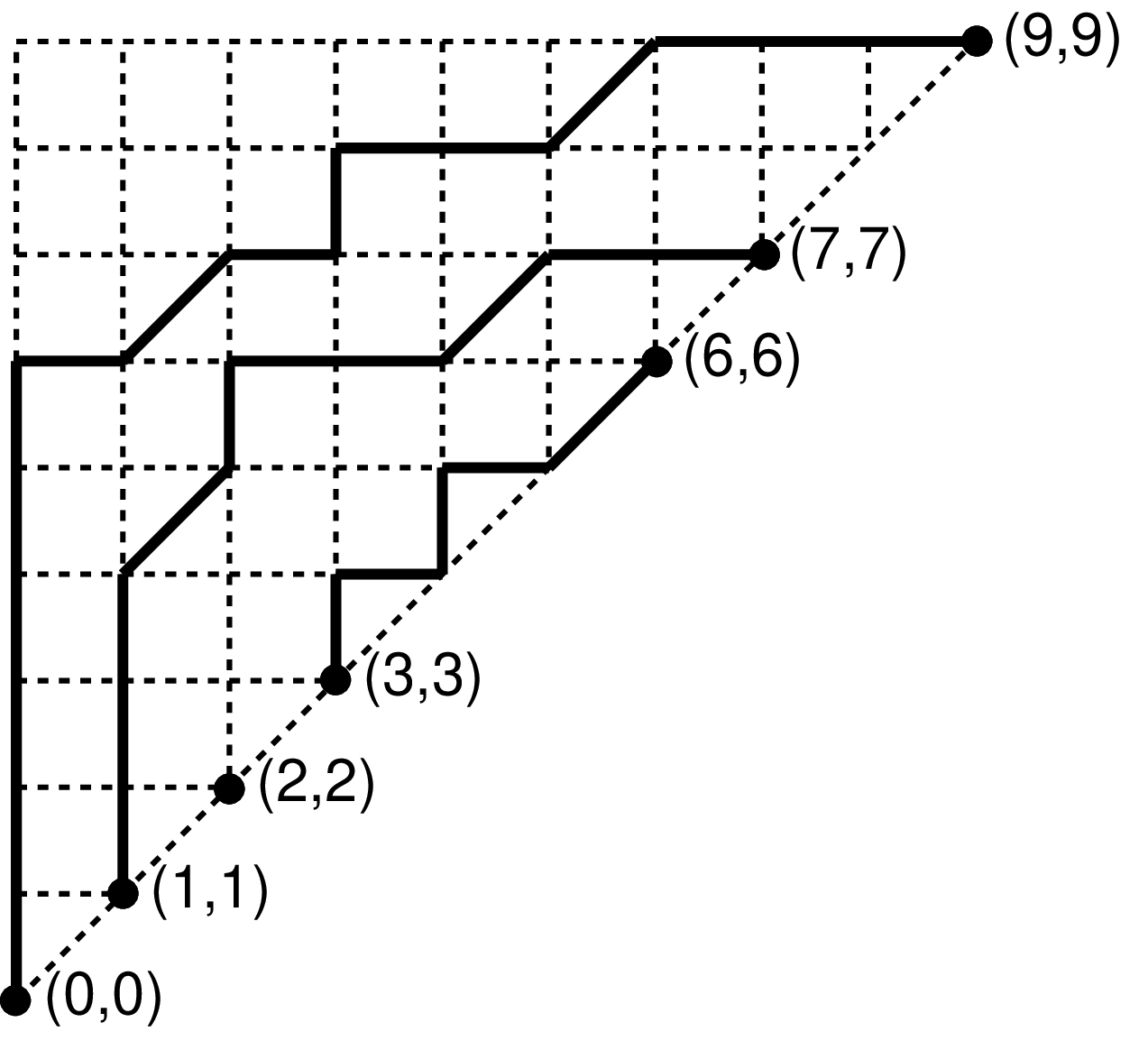}}
    \hspace{10mm}
    \subfigure[]{\label{fig:tilingpaths}\includegraphics[height=0.22\textwidth]{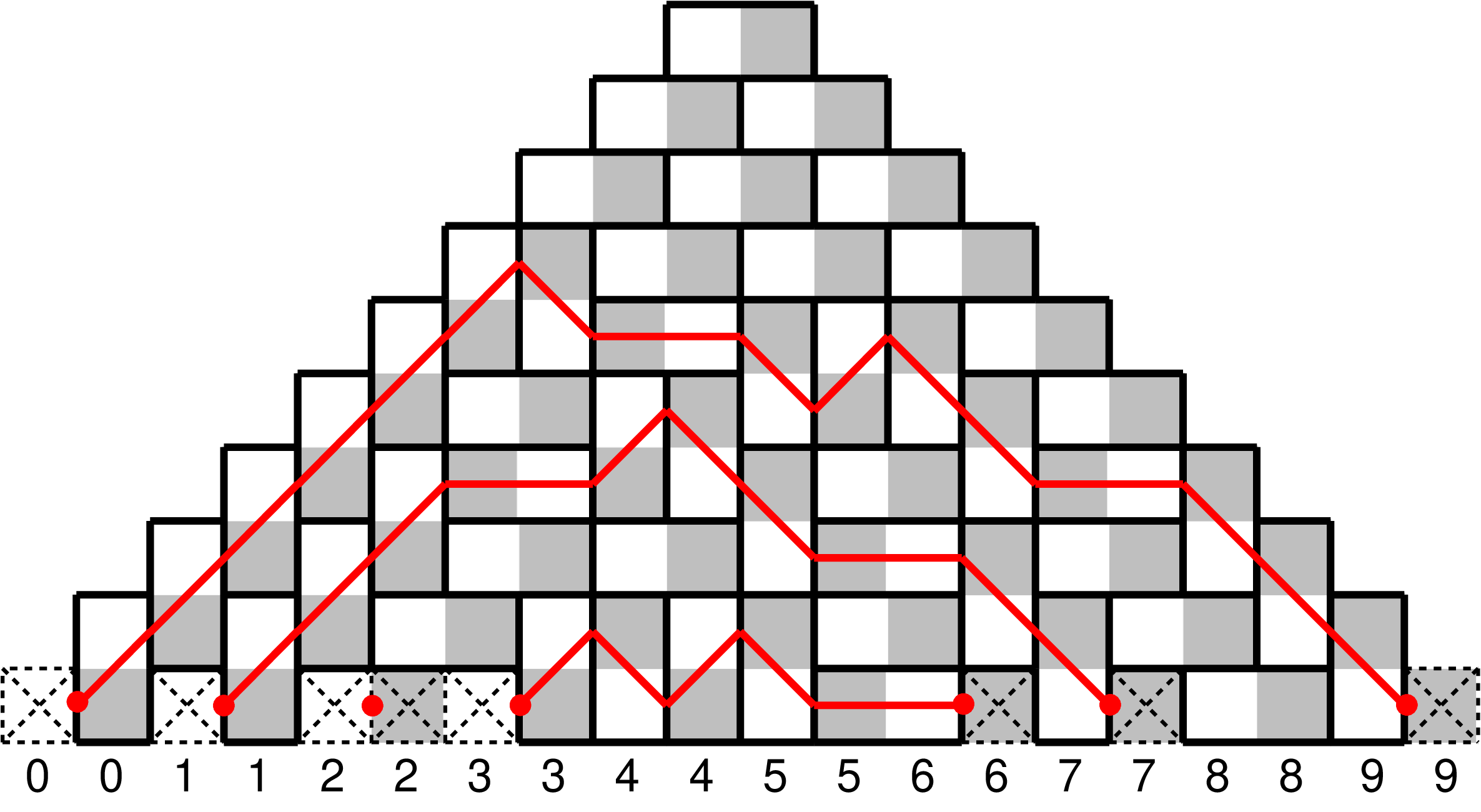}}
    \caption{(a) The decomposition of $\lambda = (4,4,3,3,3,1)$ into border strips of lengths $(n_0,n_1,n_2,n_3)=(9,6,0,3)$. (b) A $\lambda$-family of Schr\"oder paths for $\lambda = (4,4,3,3,3,1)$ with 5 diagonal steps. (c) The region $R_\lambda$ for $\lambda = (4,4,3,3,3,1)$ with the domino tiling that corresponds to the $\lambda$-family in Figure \ref{fig:paths}.}
\end{figure} 

In this paper we study the Schr\"oder generalization of these nested families of paths. Using the same notation as for the Dyck paths, a $\lambda$\textit{-family of Schr\"oder paths} is a $(k+1)$-tuple $\pi = (\pi_0,\dots,\pi_k)$ of pairwise non-intersecting lattice paths, where each $\pi_i$ is a lattice path consisting of unit north $(0,1)$, east $(1,0)$, and diagonal $(1,1)$ steps that lies weakly above the diagonal $y=x$, starts at $(j,j)$, and ends at $(j+n_j,j+n_j)$. Figure \ref{fig:paths} shows a $\lambda$-family of paths for the partition in Figure \ref{fig:lambda}. Let $\mathcal{S}_\lambda$ be the set of all $\lambda$-families of Schr\"oder paths and $\mathcal{S}_{\lambda,d}$ the set of all $\lambda$-families of Schr\"oder paths with $d$ total diagonal steps among all the paths.

We define generalized area and dinv statistics on $\mathcal{S}_{\lambda,d}$ (Section \ref{sec:statistics}) that specialize to the corresponding statistics in the cases of $\lambda$-families of Dyck paths, (single) Schr\"oder paths, and (single) Dyck paths. In the rest of the paper, we refer to generalized area and dinv statistics simply as area and dinv.

Our first result expresses $\lambda$-families of Schr\"oder paths weighted according to these statistics in the language of Macdonald polynomials. In the following formula, $\sgn(\lambda) = \pm 1$ is a sign factor introduced in \cite{LW08} whose definition is recalled in Section \ref{sec:statistics}.
\begin{theorem}\label{thm:MacdonaldIntro}
    For any partition $\lambda \vdash n$ and $0\leq d\leq n$,
    \begin{equation}\label{eq.thm1}
        \sum_{\pi\in \mathcal{S}_{\lambda,d}} q^{\area(\pi)}t^{\dinv(\pi)} = \sgn(\lambda) \langle \nabla(s_\lambda),h_d \, e_{n-d} \rangle.
    \end{equation}
\end{theorem}

Theorem \ref{thm:MacdonaldIntro} is the common generalization of several formulas relating $q,t$-Catalan combinatorics to Macdonald polynomials through the $\nabla$ operator summarized in Table \ref{tab.nabla}. When $d=0$, \eqref{eq.thm1} reduces to the formula $\sgn(\lambda)\langle \nabla(s_\lambda),e_n\rangle$ for $\lambda$-families of Dyck paths. When $\lambda =(1^n)$, one recovers the formula $\langle \nabla(e_n), h_d\,e_{n-d} \rangle$ for the $q,t$-Schr\"oder numbers. In the intersection of these cases, when $d=0$ and $\lambda = (1^n)$, the formula $\langle \nabla(e_n),e_n\rangle$ gives the $q,t$-Catalan numbers.

We obtain Theorem \ref{thm:MacdonaldIntro} from an even more general formula for the full monomial expansion of $\nabla(s_\lambda)$ in terms of \emph{labeled weakly-nested Dyck paths} introduced by Loehr and Warrington \cite{LW08} and recently proved by Blasiak, Haiman, Morse, Pun, and Seelinger \cite{BHMPS}. Our proof of Theorem \ref{thm:MacdonaldIntro} in Section \ref{sec:Mac} closely follows Haglund's argument to obtain the corresponding formula for (single) Schr\"oder paths from the shuffle formula for (single) labeled Dyck paths from the Superization section of \cite[Chapter 6]{H08}.

Our present interest in these polynomials is due to a general bijection \cite{LRS01} between domino tilings\footnote{A \emph{domino tiling} of $R$ is a covering of $R$ using dominoes ($1 \times 2$ or $2 \times 1$ rectangles) without gaps or overlaps.} of regions $R$ on the square lattice and families of non-intersecting Schr\"oder-like paths contained in $R$ (see Section \ref{sec:ADbijection}). Motivated by this bijection, we define below a region $R_\lambda$ whose domino tilings correspond to $\lambda$-families of Schr\"oder paths in the set $\mathcal{S}_{\lambda}$. This region is related to the Aztec diamond introduced by Elkies, Kuperberg, Larsen, and Propp \cite{EKLP1, EKLP2} in 1992. One goal of this project is to establish a bridge between the rich combinatorics surrounding domino tilings of the Aztec diamond (and similar regions) and $q,t$-Catalan combinatorics.

The \emph{Aztec diamond} of order $n$, denoted by $\AD_{n}$, is the union of all unit squares inside the diamond-shaped region $\{(x,y) \in \mathbb{R}^2 : |x|+|y| \leq n+1 \}$. Given a partition $\lambda$ decomposed into border strips of lengths $(n_0,\dots,n_k)$ as above, $R_\lambda$ will be defined as a certain subset of the top half, $y\geq 0$, of $\AD_{n_0+1}$. To describe this region, equip the Aztec diamond with the checkerboard coloring so that the unit squares along the top right side are colored black. Label the boxes in the bottom row of the top half of the Aztec diamond $0,0,1,1,2,2,\dots,n_0,n_0$ from left to right. The region $R_\lambda$ is defined by removing from this region the white boxes in the bottom row labeled $0,1,\dots,k$ and the black boxes in the bottom row labeled $i+n_i$ for each $i=0,\dots,k$. Let $\T(R_\lambda)$ be the set of domino tilings of $R_\lambda$. As explained in Section \ref{sec:ADbijection}, $\T(R_\lambda)$ is in bijection with $\mathcal{S}_{\lambda}$. Figure \ref{fig:tilingpaths} shows the domino tiling of $R_\lambda$ corresponding to the $\lambda$-family in Figure \ref{fig:paths}.

Passing through this bijection, we define in Section \ref{sec:ADbijection} statistics $\area$, $\dinv$, and $\diags$ directly on $\T(R_\lambda)$ that agree with the corresponding statistics on $\mathcal{S}_{\lambda,d}$. We then study the following generating polynomial of domino tilings of $R_\lambda$, where $\lambda$ is a partition of $n$:
\begin{equation}\label{eq.P-polynomial}
    P_\lambda(z;q,t) := \sum_{T\in \T(R_\lambda)} z^{\diags(T)}q^{\area(T)}t^{\dinv(T)} = \sum_{d=0}^{n}\sum_{\pi\in \mathcal{S}_{\lambda,d}} z^d q^{\area(\pi)}t^{\dinv(\pi)}.
\end{equation}

A consequence of Theorem \ref{thm:MacdonaldIntro} is the symmetry of $P_\lambda(z;q,t)$ in the variables $q$ and $t$.

\begin{corollary}\label{cor:symmetry}
    For all partitions $\lambda$, we have $P_\lambda(z;q,t)=P_\lambda(z;t,q)$.
\end{corollary}
The proof of Corollary \ref{cor:symmetry} is also found in Section \ref{sec:Mac}. As is often the case in $q,t$-Catalan combinatorics, we do not have a combinatorial proof the joint symmetry of the area and dinv statistics in general.

A special role is played by the partitions of square shape, $\lambda =(n^n)$, in which case domino tilings of $R_{\lambda}$ are in bijection with domino tilings $\AD_n$ (see Section \ref{sec:exADASM}). For convenience, we write 
\begin{equation}\label{eq.ADdef}
    \AD_n(z;q,t)= P_{(n^n)}(z;q,t).
\end{equation}
Surprisingly, the combinatorics of alternating sign matrices and domino shuffling introduced in \cite{EKLP1,EKLP2} interact well with the statistics $\area$, $\dinv$ and $\diags$ coming from $q,t$-Catalan combinatorics. We use this machinery to prove the following product formula for $\AD_n(z;q,t)$.  
\begin{theorem}\label{thm:introAD}
    When $\lambda = (n^n)$ is a partition of square shape, then
    \begin{equation}\label{eq.ADproduct}
        \AD_n(z;q,t) = (qt)^{n^2(n-1)/2}\prod_{i,j\geq0 \text{ and }i+j<n}(z+q^it^j).
    \end{equation}
\end{theorem}

Specializing all variables to $1$, we recover the well-known fact that the number of domino tilings of $\AD_n$ is $2^{n(n+1)/2}$. This product formula is similar to the formula given in \cite{EKLP1,EKLP2} weighted according to different statistics. Specifically, our statistic $\diags(T)$ is the same as their statistic $v(T)$ counting vertical dominoes up to a simple change of variables. Our $\area(T)$ statistic is similar to their rank statistic $r(T)$, although they are not the same, and our $\dinv(T)$ statistic appears to be completely new. In the recent work \cite{CGK22}, Corteel, Gitlin, and Keating introduced $k$-tilings ($k$-tuples of domino tilings) of $\AD_n$ and computed their generating polynomials. One of the statistics they considered is based on the number of ``coupled dominoes'' (see \cite[Section 2.4]{CGK22}), which is analogues to our ``domino pairs'' (see Figure \ref{fig:dominopairs}), although the two concepts are different. 

Our proof of Theorem \ref{thm:introAD} is entirely combinatorial and the formula is evidently symmetric in $q$ and $t$, so we in particular obtain a direct proof (not using Theorem \ref{thm:MacdonaldIntro}) of the $q,t$-symmetry in Corollary \ref{cor:symmetry} for partitions of square shape.

The rest of this paper is organized as follows. In Section \ref{sec:prelim}, we define statistics on $\lambda$-families of Schr\"oder paths and domino tilings, and state some background information on related combinatorics of alternating sign matrices and domino shuffling. In Section \ref{sec:Mac}, we prove Theorem \ref{thm:MacdonaldIntro} and Corollary \ref{cor:symmetry}. In Section \ref{sec:pfAD}, we discuss the special case when $\lambda=(n^n)$ and prove Theorem \ref{thm:introAD}. In Section \ref{sec:open}, we propose a conjecture regarding the divisibility of $P_{\lambda}(z;q,t)$.


\section{Preliminaries}\label{sec:prelim}

We begin this section by defining the statistics on $\lambda$-families of Schr\"oder paths $\mathcal{S}_{\lambda}$ (Section \ref{sec:statistics}). In Section \ref{sec:prelimLW}, we present the Loehr--Warrington formula which will be used to prove Theorem \ref{thm:MacdonaldIntro}. In Section \ref{sec:ADbijection}, we introduce the bijection between domino tilings of $R_{\lambda}$ and $\mathcal{S}_{\lambda}$ and then interpret these statistics defined in Section \ref{sec:statistics} in terms of dominoes. Finally, we review the relation between alternating sign matrices and domino tilings of the Aztec diamond (Section \ref{sec:ASMAD}) and the domino shuffling algorithm (Section \ref{sec:introshuffling}).

\subsection{Statistics on $\lambda$-families of Schr\"oder paths}\label{sec:statistics}

We define statistics on $\lambda$-families of Schr\"oder paths building upon the statistics for Schr\"oder paths \cite{EHKK03} and $\lambda$-families of Dyck paths \cite{LW08}. Let $\lambda= (\lambda_1,\dots,\lambda_\ell)$ be a partition. As stated in Section \ref{sec:introduction}, we decompose $\lambda$ by removing successive border strips, and let $n_j=n_j(\lambda)$ denote the number of squares in the border strip ending at the $j$th box from the right (indexed from $j=0$) in the top row.

The \emph{dinv adjustment} of $\lambda$ is defined to be the quantity $\adj(\lambda) = \sum_{j:n_j>0}(\lambda_1-1-j)$ and we set\footnote{The definitions of $\adj$ and $\sgn$ differ slightly from, but are equivalent to, those in \cite{LW08}.} $\sgn(\lambda)=(-1)^{\adj(\lambda)}$. In fact, $\adj(\lambda)$ is the total number of times a border strip in the decomposition of $\lambda$ crosses a vertical boundary of a unit square in $\lambda$. In the example $\lambda = (4,4,3,3,3,1)$ depicted in Figure \ref{fig:lambda}, $\adj(\lambda)=3+2+0 = 5$ and $\sgn(\lambda)=-1$.

As defined in the introduction, a \emph{$\lambda$-family of Schr\"oder paths} in $\mathcal{S}_{\lambda}$ is a sequence $\pi = (\pi_0,\dots,\pi_k)$ where $k=\lambda_1-1$, such that
\begin{enumerate}
    \item $\pi_j$ is a lattice path consisting of north $(0,1)$ steps, east $(1,0)$ steps, and diagonal $(1,1)$ steps starting at $(j,j)$ and ending at $(j+n_j,j+n_j)$, which never goes strictly below the diagonal $y=x$, and
    \item no two distinct paths $\pi_i$ and $\pi_j$ intersect.
\end{enumerate}
Note that the construction of the vector $(n_0,\dots,n_k)$ encoding the sizes of border strips removed from $\lambda$ ensures that there exists a $\lambda$-family of Schr\"oder paths for any partition $\lambda$. 

Given $\pi = (\pi_0,\dots,\pi_k) \in \mathcal{S}_\lambda$, we now define the first statistic
\begin{equation}\label{eq.defdiags}
    \diags(\pi) = \sum_{j=0}^{k} \diags(\pi_j),
\end{equation}
where $\diags(\pi_j)$ is the number of diagonal $(1,1)$ steps in the path $\pi_j$. We write $\mathcal{S}_{\lambda,d}$ for the subset of $\mathcal{S}_\lambda$ consisting of $\lambda$-families $\pi$ with $\diags(\pi) = d$. For example, the $\lambda$-family of Schr\"oder paths $\pi$ in Figure \ref{fig:paths} has $\diags(\pi) = 2+2+1=5.$

In order to define the other statistics on these families, we encode the paths by an array following the notation for $\lambda$-families of Dyck paths \cite{LW08}, with added decorations to encode the diagonal steps. A triangle with vertices $(c, d)$, $(c+1,d)$ and $(c+1,d+1)$ for some integers $c$ and $d$ is called an \emph{area triangle}. Given $\pi=(\pi_0,\dots,\pi_k) \in \mathcal{S}_\lambda$ and $j=0,1,\dots,k$, we define $a^{(j)}_i = a^{(j)}_i(\pi)$ for each $j\leq i\leq j+n_j-1$ to be the number of area triangles in the region bounded below by $y=i$, bounded above by $y=i+1$, bounded to the left by $\pi_j$, and bounded to the right by $y=x$. For values of $i$ outside of this range, $a^{(j)}_i$ is undefined. For each $i$ in this range, $\pi_j$ has either a single north step or a single diagonal step in the row bounded by $y=i$ and $y=i+1$. For the $i$'s for which $\pi_j$ has a diagonal step, we decorate $a^{(j)}_i$ with an overline. This collection of decorated numbers is called the \textit{area array} of $\pi$, and completely encodes the family of paths. 

We define the second statistic on $\pi = (\pi_0,\dots,\pi_k) \in \mathcal{S}_\lambda$ as follows:
\begin{equation}\label{eq.defarea}
    \area(\pi) = \sum_{j=0}^{k} \area(\pi_j),
\end{equation}
where $\area(\pi_j)$ is the total number of area triangles below $\pi_j$ and above the diagonal $y=x$. In other words, $\area(\pi)$ is the sum of all the entries in the area array of $\pi$, ignoring any decorations. 

The $\lambda$-family of Schr\"oder paths $\pi$ depicted in Figure \ref{fig:paths} has the following area array, where $\cdot$ denotes an undefined value. One can check that $\area(\pi) = 27+11+0 = 38$.
\begin{equation*}\label{eqn:areaarray}
 \begin{blockarray}{cccccccccc}
    i: & 0 & 1 & 2 & 3 & 4 & 5 & 6 & 7 & 8 \\
   \begin{block}{(cccccccccc)}
     a^{(0)}: & 0 & 1 & 2 & 3 & 4 & 5 & \overline{5} & 4 & \overline{3} \\
     a^{(1)}: & \cdot & 0 & 1 & 2 & \overline{3} & 3 & \overline{2} & \cdot & \cdot  \\
     a^{(2)}: & \cdot & \cdot & \cdot & \cdot & \cdot & \cdot & \cdot & \cdot & \cdot \\
     a^{(3)}: & \cdot & \cdot & \cdot & 0 & 0 & \overline{0} & \cdot & \cdot & \cdot  \\
   \end{block}
 \end{blockarray}
\end{equation*}

To define our last statistic, the dinv statistic, we designate certain pairs of steps in the paths $\pi \in \mathcal{S}_{\lambda}$ to be \emph{dinv pairs}, and record the total number of such pairs. It is more convenient to define these pairs directly from the area array, where we associate each entry $a^{(j)}_i$ with the north or diagonal step (depending on whether or not the entry is decorated) of the path $\pi_j$ in the region $i\leq y\leq i+1$. We say that a pair of entries $a^{(u)}_b$ and $a^{(v)}_c$ in the area array of $\pi$ is a \emph{dinv pair} in $\pi$ if either of the following two cases hold:
\begin{enumerate}
    \item[] Case 1. $a^{(u)}_b = a^{(v)}_c$ with $b<c$ and $a^{(u)}_b$ is not decorated, or
    \item[] Case 2. $a^{(u)}_b = a^{(v)}_c + 1$ with $b\leq c$ and $a^{(v)}_c$ is not decorated.
\end{enumerate}
For equalities of entries in the area array in the conditions above, we consider only the numerical values and ignore any decorations. We then define the dinv statistic of $\pi \in \mathcal{S}_{\lambda}$ to be the quantity
\begin{equation}\label{eq.defdinv}
    \dinv(\pi) = \adj(\lambda) + \left| \{ \text{dinv pairs in } \pi \} \right|.
\end{equation}
Our running example $\pi$ in Figure \ref{fig:paths} has $\dinv(\pi) = 5+29 = 34$. 

Moreover, we can interpret and visualize dinv pairs as pairs of north or diagonal steps in some, possibly different, paths in $\pi \in \mathcal{S}_{\lambda}$ as described above. Figure~\ref{fig:dinvpairs1} depicts the possible arrangements of pairs of steps that correspond to dinv pairs. The dotted lines indicate the alignment of the pairs of steps relative to some translation of the diagonal $y=x$.

\begin{figure}[h]
    \centering
    \includegraphics[scale=0.15]{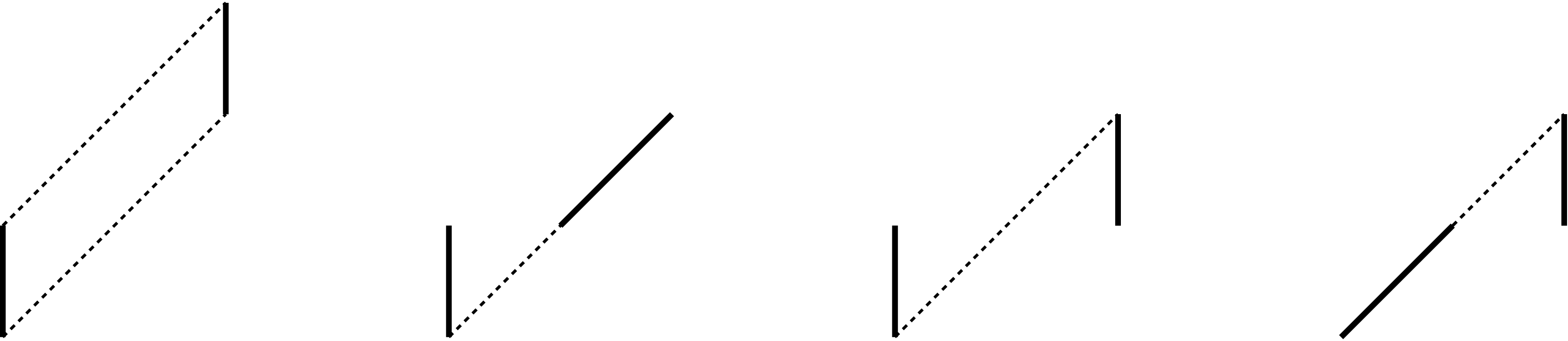}
    \caption{Dinv pairs of steps from left to right corresponding to: Case 1 with $a^{(v)}_c$ undecorated, Case 1 with $a^{(v)}_c$ decorated, Case 2 with $a^{(u)}_b$ undecorated, and Case 2 with $a^{(u)}_b$ decorated.}
    \label{fig:dinvpairs1}
\end{figure}

These statistics are the common generalization of those defined for single Schr\"oder paths and for nested Dyck paths. In the case $\lambda = (1,1,\dots,1)$, a $\lambda$-family of Schr\"oder paths reduces to a single Schr\"oder path, and the statistics we have defined above coincide with those for single Schr\"oder paths \cite{EHKK03}. On the other hand, if $\pi \in \mathcal{S}_{\lambda,d}$ with $d=0$, then it can be considered as a $\lambda$-family of Dyck paths and the area and dinv statistics above coincide with those defined in \cite{LW08}. Both of these cases specialize further to $\pi \in \mathcal{S}_{\lambda ,d}$ for $\lambda = (1,1,\dots,1)$ and $d=0$, that is, single Dyck paths, in which case these statistics recover the area and dinv statistics for the $q,t$-Catalan numbers \cite{GH02}.

\subsection{The Loehr--Warrington Formula}\label{sec:prelimLW}

The Loehr--Warrington formula gives the full monomial expansion of $\nabla(s_\lambda)$ for any partition $\lambda$ as a weighted sum over labeled, weakly-nested $\lambda$-families of Dyck paths. It was conjectured by Loehr and Warrington \cite[Conjecture~2.1]{LW08} and proved recently by Blasiak, Haiman, Morse, Pun, and Seelinger \cite{BHMPS}. We recall the definition of these objects and their statistics below informally. See \cite[Section 2.3]{LW08} for the precise definitions in terms of area and label arrays. A \emph{weakly-nested $\lambda$-families of Dyck paths} is a collection of paths $\pi = (\pi_0,\pi_1,\dots,\pi_k)$ where $k=\lambda_1-1$, such that
\begin{enumerate}
    \item Each path $\pi_j$ is a lattice path from $(j,j)$ to $(j+n_j,j+n_j)$ made up of north $(0,1)$ steps and east $(1,0)$ steps that never goes below the diagonal $y=x$,
    \item two distinct paths $\pi_i$ and $\pi_j$ can never cross and can never share east steps but may touch and may share north steps, and
    \item no path $\pi_j$ meets the starting point of any other path $\pi_i$.
\end{enumerate}

A \emph{labeling} of a weakly-nested $\lambda$-family of Dyck paths is an assignment of a positive integer to each north step in all the paths of the family such that
\begin{enumerate}
    \item On any sequence of consecutive north steps of any path $\pi_j$, the labels on these north steps strictly increase from bottom to top, and

    \item whenever there is a north steps labeled $r$ directly below a north step labeled $r'$ (i.e. the ending point of the north step labeled $r$ is the starting point of the north step labeled $r'$) on some paths $\pi_i$ and $\pi_j$ respectively with $j<i$, we have $r'\leq r$.
\end{enumerate}

Let $\LNDP_\lambda$ denote the set of all pairs $(\pi,w)$ where $\pi$ is a weakly-nested $\lambda$-family of Dyck paths with labeling $w$. An example of a labeled weakly-nested $\lambda$-family of Dyck paths is shown in Figure \ref{fig:labeledfamily}, with labels on the outermost path appearing on the left of the corresponding edge, and labels on the other paths on the right.

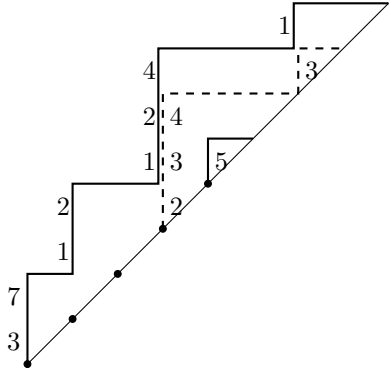
\begin{figure}[h]
    \centering
    \begin{tikzpicture}[scale=0.6]

        \draw (0,0)--(8,8);
        \draw[fill=black] (0,0) circle (.5ex);
        \draw[fill=black] (1,1) circle (.5ex);
        \draw[fill=black] (2,2) circle (.5ex);
        \draw[fill=black] (3,3) circle (.5ex);
        \draw[fill=black] (4,4) circle (.5ex);
        
        \draw[thick] (0,0)--(0,2)--(1,2)--(1,4)--(2.9,4)--(2.9,7)--(5.9,7)--(5.9,8)--(8,8);
        \draw[thick,dashed] (3,3)--(3,6)--(6,6)--(6,7)--(7,7);
        \draw[thick] (4,4)--(4,5)--(5,5);
    
        \node at (-.3,.5) {3};
        \node at (-.3,1.5) {7};
        \node at (.8,2.5) {1};
        \node at (.8,3.5) {2};
        \node at (2.7,4.5) {1};
        \node at (2.7,5.5) {2};
        \node at (2.7,6.5) {4};
        \node at (5.7,7.5) {1};
    
        \node at (3.3,3.5) {2};
        \node at (3.3,4.5) {3};
        \node at (3.3,5.5) {4};
        \node at (6.3,6.5) {3};
    
        \node at (4.3,4.5) {5};
    \end{tikzpicture}
    \caption{A labeled weakly-nested $(5,3,3,2)$-family of Dyck paths.}
    \label{fig:labeledfamily}
\end{figure}

The area of a Dyck path is the number of complete unit boxes (or equivalently the number of area triangles, as defined in Section \ref{sec:statistics}) underneath the path and above the diagonal $y=x$. The area of a weakly-nested $\lambda$-family of Dyck paths is the sum of the areas of the paths in the family. For a family $\pi$ with labeling $w$, the statistic $\area(\pi,w)$ is defined as the area of the family $\pi$, not depending on the labeling of the family. For example, the area of the $\lambda$-family in Figure \ref{fig:labeledfamily} is $14$.

The dinv statistic of a labeled family, however, does depend on the labeling. Let $(\pi,w) \in \LNDP_{\lambda}$ and $N_1$ and $N_2$ be any two north steps from any paths in $\pi$ with distinct labels $L_1$ and $L_2$ respectively, where $L_1<L_2$. We say that this pair is a \textit{dinv pair} for $(\pi,w)$ if either of the following conditions hold:
\begin{enumerate}
    \item The north steps are on the same diagonal (some line $y=x+c$ meets the bottom-most points of both $N_1$ and $N_2$) and $N_2$ lies weakly to the left of $N_1$.
    \item There is some line $y=x+c$ that meets the top-most point of $N_2$ and the bottom-most point of $N_1$, and $N_2$ lies strictly to the left of $N_1$.
\end{enumerate}
Figure \ref{fig:dinvpairs} shows two families of pairs of vertical steps that are dinv pairs for any labeling with $a<b$.
\begin{figure}[h]
    \centering
    \begin{tikzpicture}[scale=0.5]

        \draw[thick] (0,0)--(0,1);
        \draw[thick] (.1,0)--(.1,1);
    
        \node at (-.5,.5) {$a$};
        \node at (.5,.5) {$b$};
    \end{tikzpicture}
    \begin{tikzpicture}[scale=0.5]

        \draw[thick] (0,0)--(0,1);
        \draw[dashed,thin] (0,0)--(1,1);
        \draw[dashed,thin] (0,1)--(1,2);
        \draw[thick] (1,1)--(1,2);
    
        \node at (-.5,.5) {$a$};
        \node at (1.5,1.5) {$b$};
    \end{tikzpicture}
    \begin{tikzpicture}[scale=0.5]

        \draw[thick] (0,0)--(0,1);
        \draw[dashed,thin] (0,0)--(2,2);
        \draw[dashed,thin] (0,1)--(2,3);
        \draw[thick] (2,2)--(2,3);
    
        \node at (-.5,.5) {$a$};
        \node at (2.5,2.5) {$b$};
    \end{tikzpicture}
    \begin{tikzpicture}[scale=0.5]

        \draw[thick] (0,0)--(0,1);
        \draw[dashed,thin] (0,0)--(3,3);
        \draw[dashed,thin] (0,1)--(3,4);
        \draw[thick] (3,3)--(3,4);
    
        \node at (-.5,.5) {$a$};
        \node at (3.5,3.5) {$b$};
        \node at (4.5,1.5) {$\cdots$}; 
    \end{tikzpicture}
    
    \begin{tikzpicture}[scale=0.5]

        \draw[thick] (0,0)--(0,1);
        \draw[dashed,thin] (0,0)--(1,1);
        \draw[thick] (1,0)--(1,1);
    
        \node at (-.5,.5) {$b$};
        \node at (1.5,.5) {$a$};
    \end{tikzpicture}
    \begin{tikzpicture}[scale=0.5]

        \draw[thick] (0,0)--(0,1);
        \draw[dashed,thin] (0,0)--(2,2);
        \draw[thick] (2,1)--(2,2);
    
        \node at (-.5,.5) {$b$};
        \node at (2.5,1.5) {$a$};
    \end{tikzpicture}
    \begin{tikzpicture}[scale=0.5]

        \draw[thick] (0,0)--(0,1);
        \draw[dashed,thin] (0,0)--(3,3);
        \draw[thick] (3,2)--(3,3);
    
        \node at (-.5,.5) {$b$};
        \node at (3.5,2.5) {$a$};
        \node at (4.5,1.5) {$\cdots$}; 
    \end{tikzpicture}
    \caption{Dinv pairs for any labeling with $a<b$}
    \label{fig:dinvpairs}
\end{figure}
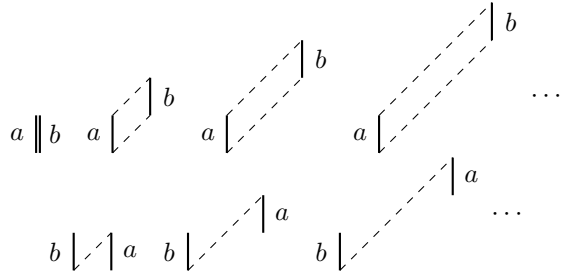

For any pair $(\pi,w) \in \LNDP_{\lambda}$, define the statistic $\dinv(\pi,w)=\adj(\lambda)+ \left| \{\text{dinv pairs in } (\pi,w)\} \right|$. For example, the labeled family in Figure \ref{fig:labeledfamily} has $\dinv(\pi,w)=5+14=19$. Finally, let $x^w$ denote the monomial in the variables $\{x_i\}_{i \geq 1}$ such that the exponent on each $x_i$ is the number of $i$'s that appear in the labeling $w$. 

The following theorem presents the Loehr--Warrington formula, it will be used to prove Theorem \ref{thm:MacdonaldIntro} later in Section \ref{sec:Mac}. 
\begin{theorem}[\cite{LW08} and \cite{BHMPS}]\label{thm:LW}
    For any partition $\lambda$,
    \begin{equation}\label{eq.LWformula}
        \nabla(s_\lambda) = \sgn(\lambda)\sum_{(\pi,w)\in \LNDP_\lambda} q^{\area(\pi,w)}t^{\dinv(\pi,w)}x^w.
    \end{equation}
\end{theorem}

\subsection{Bijection between Domino Tilings and Schr\"oder-like Paths}\label{sec:ADbijection}

In general, there is a bijection \cite{LRS01} between domino tilings of a region $R$ on the square lattice and families of non-intersecting paths contained in $R$. In particular, domino tilings of $R_\lambda$ are in bijection with $\lambda$-families of Schr\"oder paths. This bijection is described below.

Given a domino tiling of $R_\lambda$, the corresponding family of paths is given by decorating each of the four types of dominoes that can appear relative to the checkerboard coloring with a (possibly empty) path step as shown in Figure \ref{fig:key}. Regarding each of these path steps as oriented from left to right, each step enters a domino in a black square and exits a domino in a white square. The union of all of these individual path steps therefore forms a collection of non-intersecting paths in $R_\lambda$. The starting points of these paths are the black squares in $R_\lambda$ such that the white square immediately to its left is not in $R_\lambda$ and the ending points of these paths are the white squares in $R_\lambda$ such that the black square immediately to its right is not in $R_\lambda$. 
\begin{figure}[hbt!]
    \centering
    \includegraphics[scale=0.28]{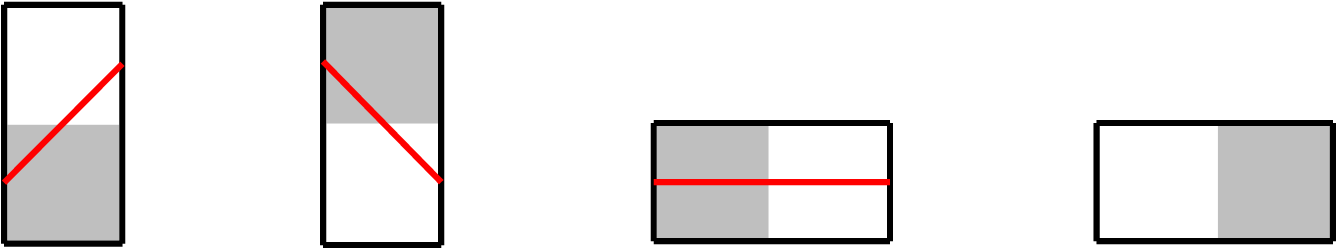}
    \caption{The four types of dominoes and their corresponding path steps.}
    \label{fig:key}
\end{figure} 

One can view the resulting family of paths as $\lambda$-family of Schr\"oder paths by a change of coordinates that amounts to rotating $45^\circ$ counterclockwise and rescaling so that the longest path begins at $(0,0)$ and ends at $(n_0,n_0)$. More precisely, the corresponding $\lambda$-family of Schr\"oder paths is the image of the paths in $R_\lambda$ under the unique affine transformation that takes the points $(-n_0,1/2), (0,n_0+1/2),$ and $(n_0,1/2)$ to $(0,0), (0,n_0),$ and $(n_0,n_0)$, respectively. See Figures \ref{fig:paths} and \ref{fig:tilingpaths}. Conversely, given a $\lambda$-family of Schr\"oder paths, we can undo the affine transformation above to obtain a family of paths in $R_\lambda$. We then cover $R_{\lambda}$ with dominoes based on the four configurations in Figure \ref{fig:key}. This gives the corresponding domino tiling of $R_{\lambda}$.

Next, we explain how to express the statistics defined in Section \ref{sec:statistics} from the viewpoint of domino tilings of $R_{\lambda}$. Given a domino tiling $T \in \T(R_{\lambda})$ and its corresponding $\lambda$-family of Schr\"oder paths $\pi \in \mathcal{S}_{\lambda}$. Let $V(T)$ (resp., $H(T)$) be the collection of all vertical (resp., horizontal) dominoes in which the bottom (resp., left) unit square is colored black. These two particular types of domino correspond to the up and diagonal steps in $\pi$. For convenience, we also write $V^c(T)$ (resp., $H^c(T)$) for the collection of all vertical (resp., horizontal) dominoes in which the top (resp., right) unit square is colored black.

One can easily see that the diags statistic is given by $\diags(\pi) = \diags(T) = |H(T)|$.

Recall that the area statistic counts the total number of area triangles of each path in $\pi$. This can be expressed as a sum over the ``area'' of each vertical domino in $V(T)$ and horizontal domino in $H(T)$, where each area contribution depends on the $y$-coordinate of that domino. To be more precise, for a vertical domino $d_v \in V(T)$ and a horizontal domino $d_h \in H(T)$ whose bottom edge lies on $y=i$, $(0 \leq i \leq n_0)$, define $\area(d_v) = \area(d_h) = i$. Then the area statistic can be reinterpreted as
\begin{equation}\label{eq.def-domino-area}
    \area(\pi) = \area(T) = \sum_{d_v \in V(T)} \area(d_v) + \sum_{d_h \in H(T)} \area(d_h).
\end{equation}

The dinv pairs defined for $\lambda$-families of Schr\"oder paths (Figure \ref{fig:dinvpairs1}) are equivalent to the domino pairs shown in Figure \ref{fig:dominopairs}. Hence, the dinv statistic can be expressed as
\begin{equation}\label{eq.def-domino-dinv}
    \dinv(\pi) = \dinv(T) = \adj(\lambda) + \left| \{ \text{domino pairs in } T \} \right|.
\end{equation}
The number of domino pairs in $T$ can be expressed as a sum over the ``dinv'' of each vertical domino $d_v$ in $V(T)$, there are two ways to define the dinv statistic on $d_v$ based on the reference vertical domino (type I or type II) shown in each configuration of domino pairs (Figure \ref{fig:dominopairs}). 
\begin{figure}[hbt!]
    \centering
    \subfigure[]{\label{fig:dominopairs}\includegraphics[height=0.13\textwidth]{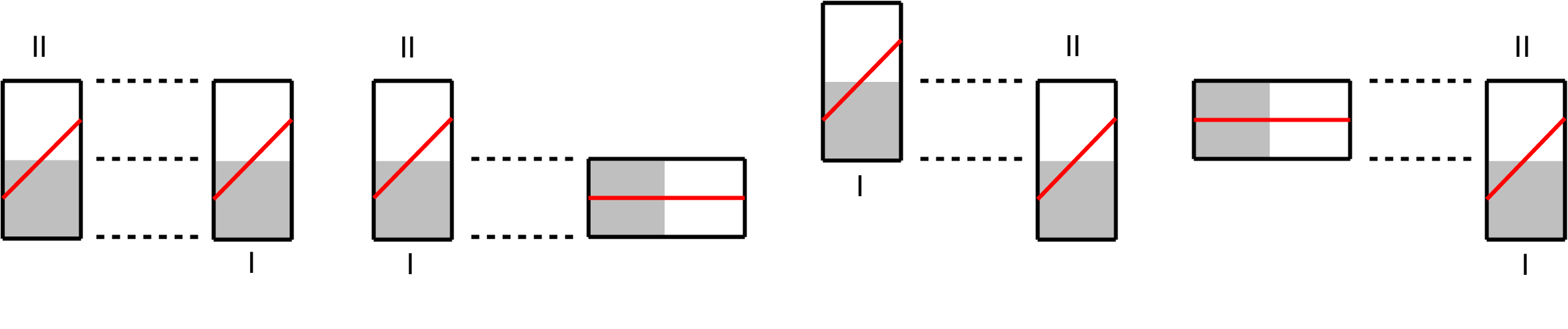}}
    \subfigure[]{\label{fig:dominotypeI}\includegraphics[width=0.45\textwidth]{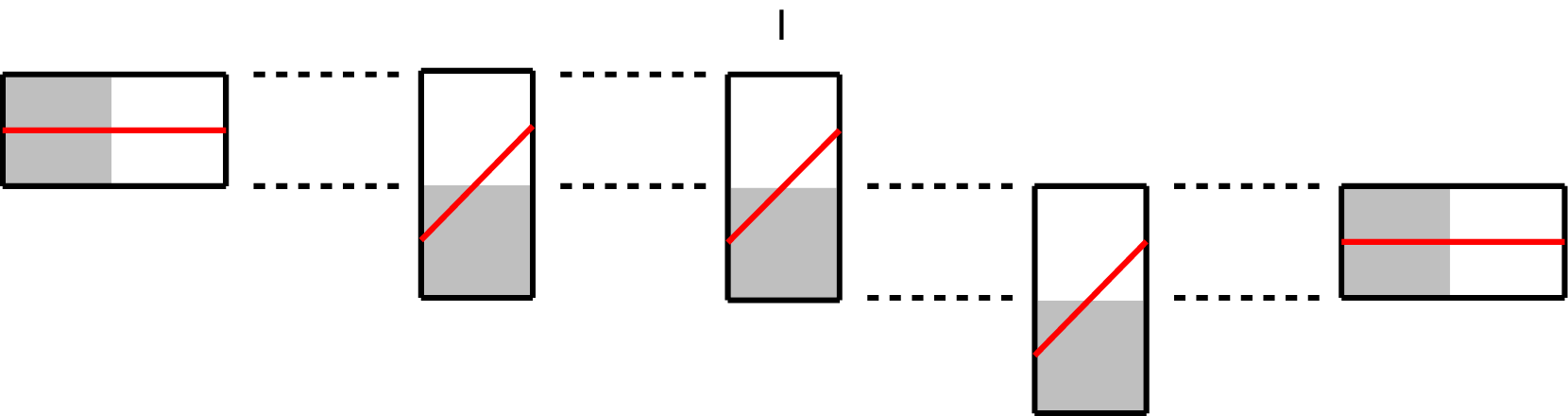}}
    \hspace{7mm}
    \subfigure[]{\label{fig:dominotypeII}\includegraphics[width=0.45\textwidth]{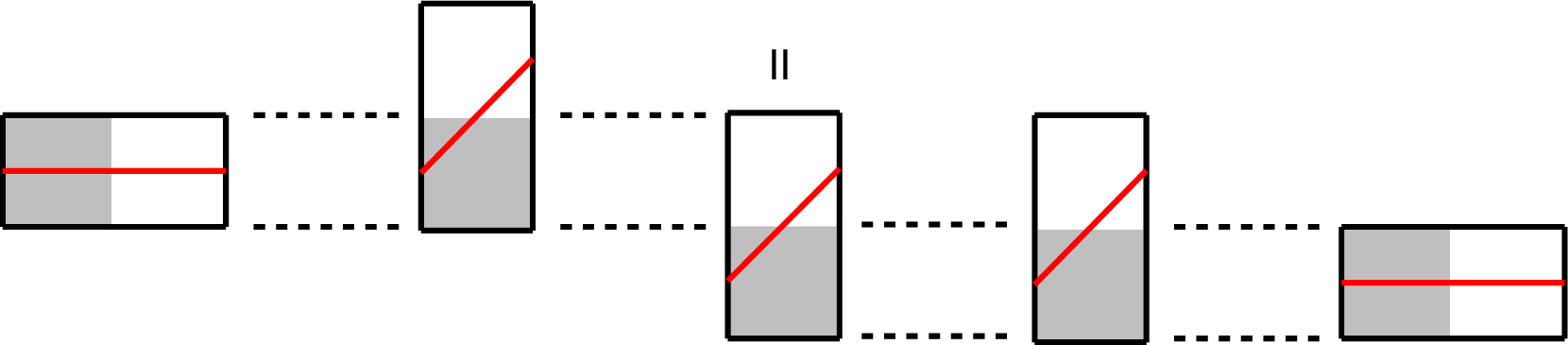}}
    \caption{(a) The four configurations of domino pairs are shown, with dotted lines indicating the alignment of the dominoes. The two types of reference dominoes are labeled as I and II. (b) The type I reference domino and its alignment within the four configurations of domino pairs. (c) The type II reference domino and its alignment within the four configurations of domino pairs.}\label{fig:dp}
\end{figure}

For the type I reference domino $d_v$, we define $\dinv_1(d_v)$ to be the number of these four kinds of dominoes (two on the left and two on the right of $d_v$) aligning in the way shown in Figure \ref{fig:dominotypeI}. To the left of $d_v$, we count the number of horizontal dominoes in $H(T)$ aligning with the top unit square of $d_v$, and vertical dominoes in $V(T)$ whose bottom unit square is aligned with the bottom unit square of $d_v$; to the right of $d_v$, we count the number of horizontal dominoes in $H(T)$ aligning with the bottom unit square of $d_v$, and vertical dominoes in $V(T)$ whose top unit square is aligned with the bottom unit square of $d_v$. Then $\dinv_1(d_v)$ is the sum of these two quantities mentioned above. Summing $\dinv_1(d_v)$ over all the vertical dominoes in $V(T)$ gives the number of domino pairs in $T$. 

Similarly, for the type II reference domino $d_v$, $\dinv_2(d_v)$ is defined to be the number of these four kinds of dominoes aligning in the way shown in Figure \ref{fig:dominotypeII}. The only difference between these two types is the alignment of the vertical dominoes to the left and the right of $d_v$. Overall, we can write the number of domino pairs in $T$ as
\begin{equation}\label{eq.dominopairs}
    \left| \{ \text{domino pairs in } T \} \right| = \sum_{d_v \in V(T)} \dinv_1(d_v) = \sum_{d_v \in V(T)} \dinv_2(d_v).
\end{equation}
These two expressions will be helpful to compute the dinv statistic later in Section \ref{sec:pfAD}.


\subsection{Alternating Sign Matrices and Aztec Diamonds}\label{sec:ASMAD}

An \emph{alternating sign matrix} (ASM) of order $n$ is an $n \times n$ matrix with entries $0, 1$ or $-1$ such that all row and column sums are equal to $1$ and the non-zero entries alternate in sign in each row and column. Mills, Robbins, and Rumsey \cite{MRR83} introduced them in the early 1980s. We write $\ASM_n$ for the set of ASMs of order $n$. The connection between ASMs and domino tilings of the Aztec diamond was initially discovered in \cite{EKLP1, EKLP2}. They showed that the number of domino tilings of the Aztec diamond can be expressed as a weighted enumeration, known as the \emph{$2$-enumeration} of ASMs based on the number of $1$'s or $-1$'s in a matrix. The weighted enumeration is given by
\begin{equation}\label{eq.2-enumasm}
  |\T(\AD_n)| = \sum_{M \in \ASM_n}2^{\mathsf{N}_{+}(M)} = \sum_{M \in \ASM_{n+1}}2^{\mathsf{N}_{-}(M)},
\end{equation}
where $\mathsf{N}_{+}(M)$ (resp., $\mathsf{N}_{-}(M)$) is the number of $1$'s (resp., $-1$'s) in the matrix $M$.
This connection was made explicit later by Ciucu \cite[Section 2]{Ciucu96}\footnote{In \cite[Section 2]{Ciucu96}, this connection was written in the equivalent language of perfect matchings on the planar dual graph of $\AD_n$.}, the idea is stated as follows. 

For two integers $x$ and $y$, we say a (lattice) point $(x,y)$ in $\AD_n$ is \emph{even} (resp., \emph{odd}) if $x+y$ and $n$ have the same (resp., opposite) parity. For each domino tiling $T$ of $\AD_n$, we first add four unit segments connecting points $(0,n)$ with $(0,n+1)$, $(-n,0)$ with $(-(n+1),0)$, $(0,-n)$ with $(0,-(n+1))$, and $(n,0)$ with $(n+1,0)$. Second, we define the maps $M_e$ and $M_o$ below.
\begin{equation}\label{eq.ASMAD}
    M_e(v_e) = \begin{cases}
         1 & \text{if $\deg(v_e)=4$,} \\
         0 & \text{if $\deg(v_e)=3$,} \\
        -1 & \text{if $\deg(v_e)=2$,}
          \end{cases} \text{\quad and \quad} 
    M_o(v_o) = \begin{cases}
         1 & \text{if $\deg(v_o)=2$,} \\
         0 & \text{if $\deg(v_o)=3$,} \\
        -1 & \text{if $\deg(v_o)=4$,}
          \end{cases}
\end{equation}
where $\deg(\cdot)$ is the degree of an even point $v_e$ or odd point $v_o$ by viewing a domino tiling as a graph. Finally, the image of all even (resp., odd) points under the map $M_e$ (resp., $M_o$) forms a square array (which is an ASM) of size $n+1$ (resp., $n$); see Figures \ref{fig:ASMAD1} and \ref{fig:ASMAD2}. We write $M_e(T)$ and $M_o(T)$ for such ASMs that correspond to $T$.
\begin{figure}[h]
    \centering
    \subfigure[]{\label{fig:ASMAD1}\includegraphics[height=0.25\textwidth]{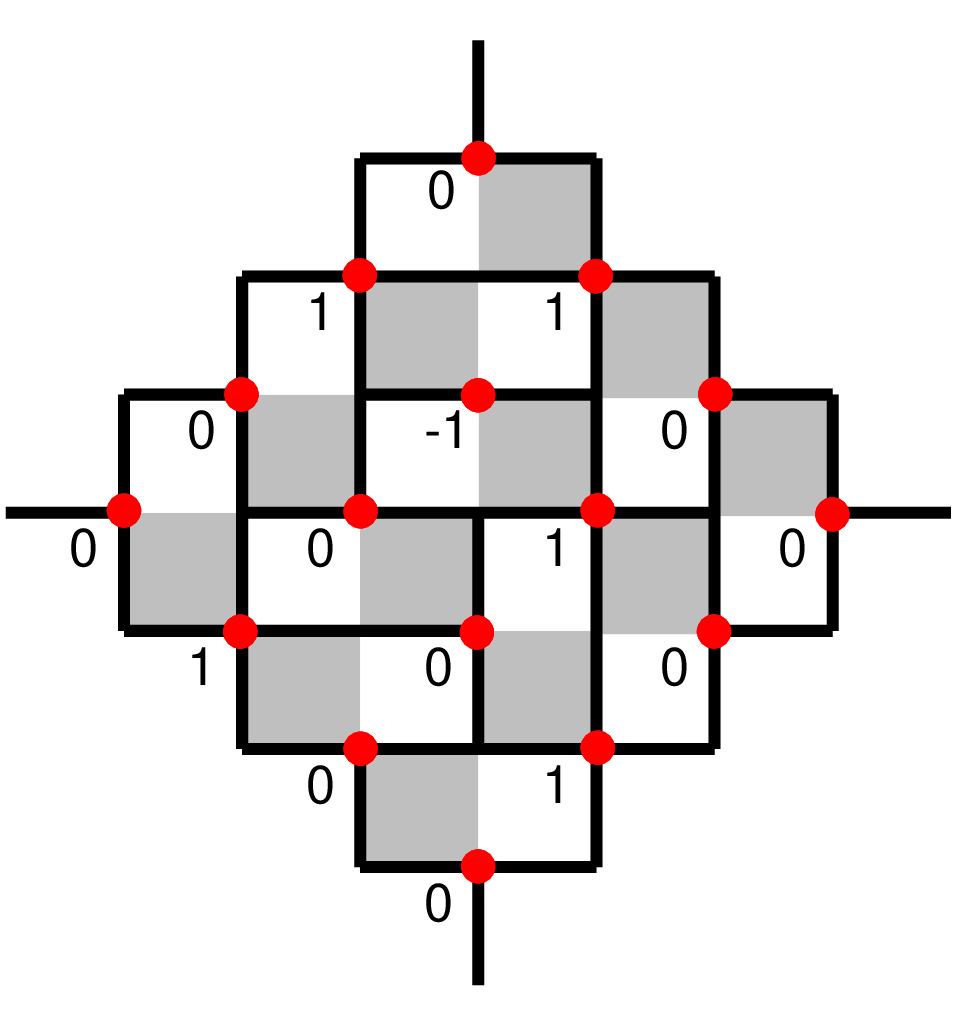}}
    \hspace{10mm}
    \subfigure[]{\label{fig:ASMAD2}\includegraphics[height=0.25\textwidth]{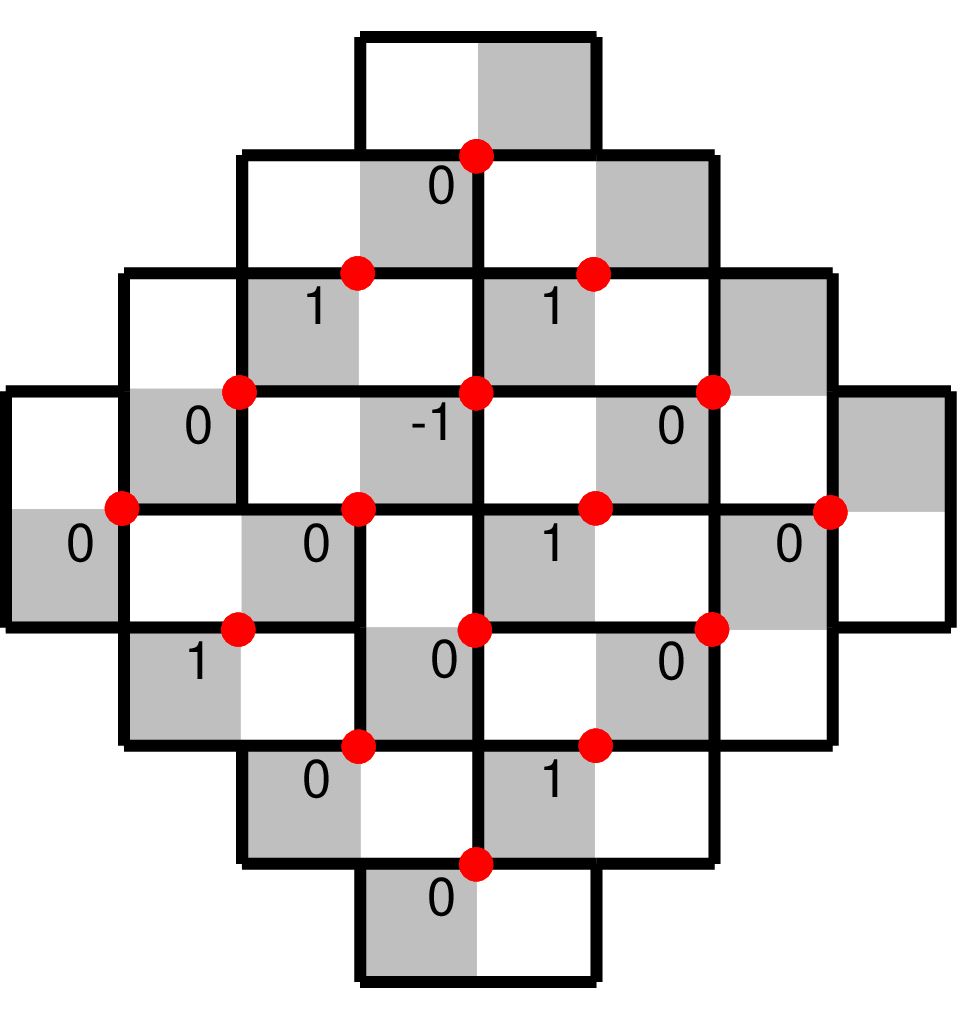}}
    \caption{(a) An example of $M_e(T)$ where $T \in \T(\AD_3)$, even points are marked in red. (b) An example of $M_o(T^{\prime})$ where $T^{\prime} \in \T(\AD_4)$, odd points are marked in red.}
\end{figure}

\subsection{Domino shuffling}\label{sec:introshuffling}

Domino shuffling is first presented in \cite{EKLP1,EKLP2} and is used to inductively enumerate domino tilings of $\AD_n$. It is a process that involves moving dominoes in a specific way, we follow closely their idea and summarize some facts about the domino shuffling below.

Recall that we equip $\AD_n$ the checkerboard coloring so that the unit squares along the top right side are black. Given a domino tiling $T$ of $\AD_n$, if we can find a $2 \times 2$ block in $T$ (consisting of two horizontal or two vertical dominoes), then we say this $2 \times 2$ block is \emph{even} (resp., \emph{odd}) if the top right unit square in this block is black (resp., white). An \emph{odd-deficient} (resp., \emph{even-deficient}) tiling of $T$ is obtained from $T$ by removing all odd (resp., even) blocks of $T$. 

Domino shuffling is a map $S$ which sends an odd-deficient tiling $T_o$ of order $n$ to an even-deficient tiling $T_e = S(T_o)$ of order $n+1$ by the following actions: (1) slide a vertical domino in which the bottom square is black of $T_o$ one unit to its left, (2) slide a vertical domino in which the top square is black of $T_o$ one unit to its right, (3) slide a horizontal domino in which the left square is black of $T_o$ one unit above it, and (4) slide a horizontal domino in which the right square is black of $T_o$ one unit below it. We should emphasize that the colors are moved along with dominoes in our case (in \cite{EKLP1,EKLP2}, the colors are swapped after shuffling, so our terminologies are slightly different from theirs). Figure \ref{fig:ADTeven} gives an example of an even-deficient tiling $T_e$ which is obtained from the odd-deficient tiling $T_o$ shown in Figure \ref{fig:ADTodd} under the domino shuffling.

Suppose that a domino tiling $T$ of $\AD_n$ contains $m$ odd blocks. The process of sliding dominoes from its odd-deficient tiling $T_o$ described above opens up some spaces in the Aztec diamond region, these spaces can be filled by $m+n+1$ even blocks. In other words, $T_e$ has $m+n+1$ even blocks. One can see that the center of an odd (resp., even) block is assigned $-1$ (resp., $1$) under the map $M_e$ (resp., $M_o$) described in Section \ref{sec:ASMAD} since these centers are of degree $2$. 

Moreover, we define $M_e(T_o)$ to be the ASM by sending even points of $T_o$ via the map $M_e$ and assigning the center of odd blocks to $-1$. Similarly, define $M_o(T_e)$ to be the ASM by sending odd points of $T_e$ via the map $M_o$ and assigning the center of even blocks to $1$. There is a fact that the ASMs $M_e(T_o)$ and $M_o(T_e)$ are identical, where $T_e$ is obtained from $T_o$ under the domino shuffling. Figure \ref{fig:asm2} shows the ASM $M_e(T_o)$ and $M_o(T_e)$, where $T_o$ is the odd-deficient tiling in Figure \ref{fig:ADTodd} while $T_e$ is the even-deficient tiling in Figure \ref{fig:ADTeven}. One can see that the $-1$ corresponds to the odd block in Figure \ref{fig:ADTodd} and the five $1$'s correspond to the five even blocks in Figure \ref{fig:ADTeven}.
\begin{figure}[hbt!]
    \centering
    \hspace{7mm}
    \subfigure[]{\label{fig:ADTodd}\includegraphics[height=0.15\textwidth]{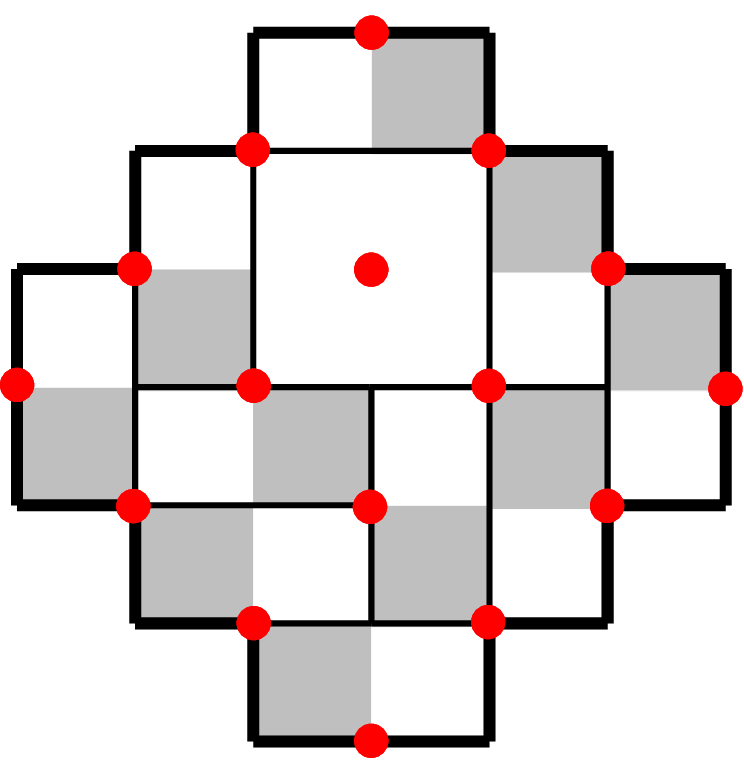}}
    \hspace{15mm}
    \subfigure[]{\label{fig:ADTeven}\includegraphics[height=0.20\textwidth]{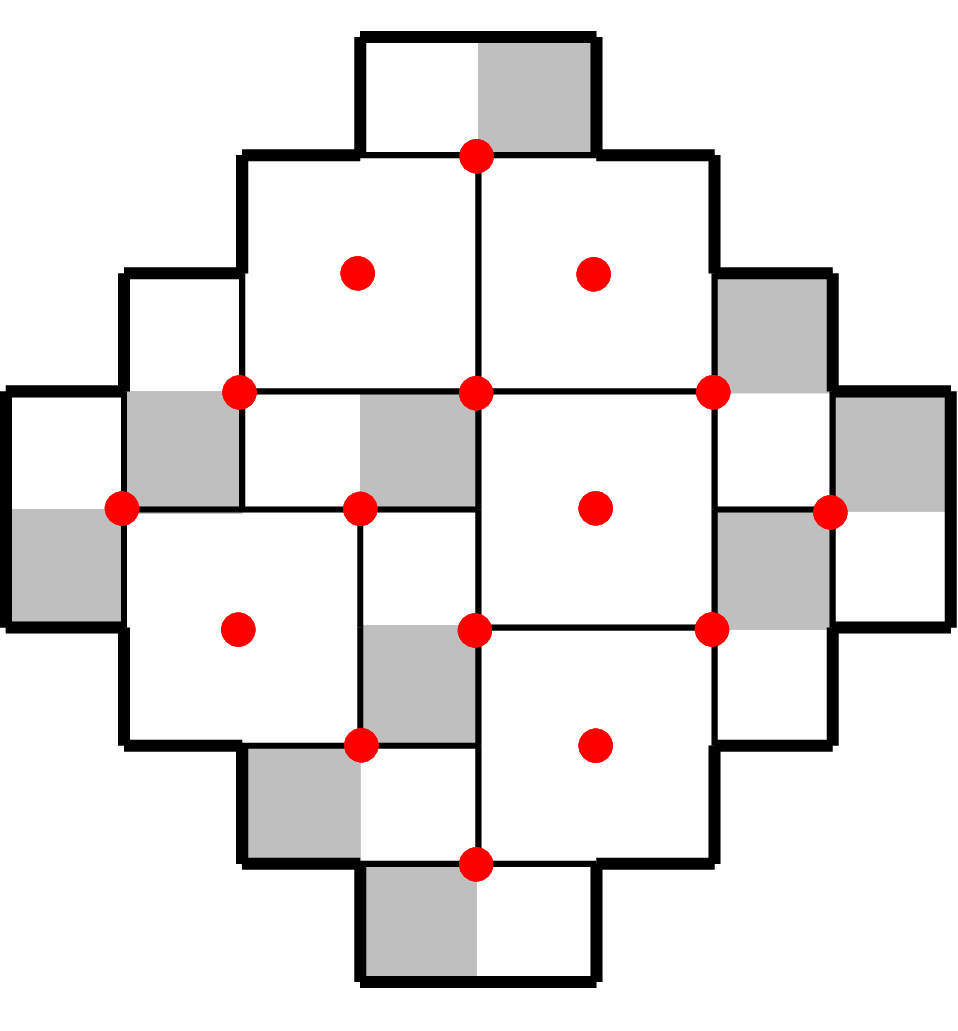}}
    \hspace{15mm}
    \subfigure[]{\label{fig:asm2}\includegraphics[height=0.175\textwidth]{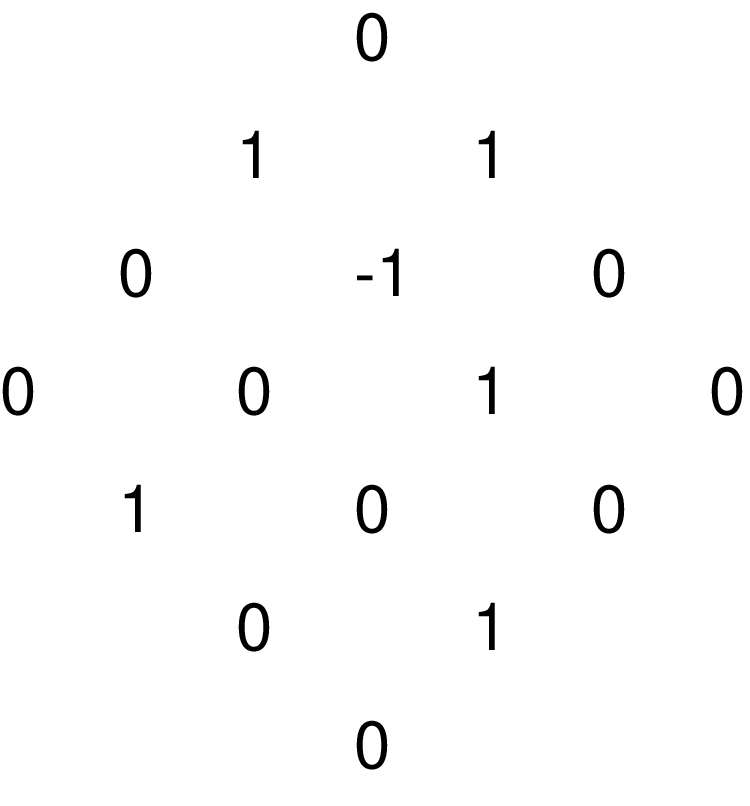}}
    \caption{(a) An odd-deficient tiling $T_o$ of the domino tiling given in Figure \ref{fig:ASMAD1}, even points are marked in red. (b) The even-deficient tiling $T_e$ under the domino shuffling of $T_o$ shown in Figure \ref{fig:ADTodd}, odd points are marked in red. (c) The ASM $M_e(T_o) = M_o(T_e)$.}\label{fig.shuffleexample}
\end{figure}
%


\section{Proof of the Nabla Operator Formula}\label{sec:Mac}

In this section, we first deduce Theorem \ref{thm:MacdonaldIntro} from the Loehr--Warrington formula (as shown in Theorem \ref{thm:LW}). The method of proof is analogous to the deduction of the $q,t$-Schr\"oder formula from the shuffle formula (see \cite[Exercise 6.15]{H08}). We then proof the $q,t$-symmetry of $P_{\lambda}(z;q,t)$ (as shown in Corollary \ref{cor:symmetry}).

For a weakly-nested $\lambda$-family of Dyck paths $\pi = (\pi_i)$, let $L(\pi)$ denote the set of all labelings of $\pi$ as in Section \ref{sec:prelimLW} and define
\begin{equation}\label{eq.Fpi}
    F_{\pi} = \sum_{w\in L(\pi)} x^w t^{\dinv(\pi,w)}.
\end{equation}
It has been shown that $F_\pi$ is a symmetric function in $\{x_i\}_{i \geq 1}$ by writing as a sum of LLT functions (\cite[Theorem 2.2]{LW08}). By Theorem \ref{thm:LW}, $F_{\pi}$ satisfies $\nabla(s_\lambda) = \sgn(\lambda)\sum_\pi q^{\area(\pi)}F_\pi$, where the sum is over all weakly-nested $\lambda$-families of Dyck paths. 

The \emph{weight} of a labeling $w$ of $\pi$ is the composition $\wt(w) = (\wt(w)_1,\wt(w)_2,\dots)$ where $\wt(w)_i$ is the number of $i$'s in the labeling $w$. We say that a labeling $w$ of $\pi$ is called \emph{standard} if the labels are the numbers $1,2,\dots,|\lambda|$ in some order with no repetitions, that is, $\wt(w) = (1,\dots,1,0,\dots)$. Let $SL(\pi)$ denote the set of standard labelings of $\pi$. The \emph{reading word} of $w$, denoted $\Read(w)$, is defined by reading the labels appearing in $w$ along diagonals from top to bottom starting with the diagonal furthest from $y=x$. For labels on coinciding edges, break ties by reading the labels from the largest indexed path to the smallest, that is, from the weakly innermost path to the weakly outermost path. For example, the reading word of the labeling in Figure \ref{fig:labeledfamily} is $\Read(w)=4422131173523$.

Given a labeling $w\in L(\pi)$ (not necessarily standard), the \emph{standardization} of $\Read(w)$ is defined as follows: Replace all the $1$'s in $\Read(w)$ with the digits $1,2,\dots,\wt(w)_1$ in the order they appear in $\Read(w)$. Replace all the $2$'s in $\Read(w)$, not including any change in the previous steps of standardization, with the digits $\wt(w)_1+1,\dots,\wt(w)_1+\wt(w)_2$ in order. Continuing in this way produces a word in which each digit $1,2,\dots,|\lambda|$ appears once each, which we call the standardized reading word of $w$, $\Read'(w)$.

For example, the standardization of $\Read(w) = 4422131173523$, corresponding to the labeling in Figure \ref{fig:labeledfamily} is  $\Read(w') = 10,11,4,5,1,7,2,3,13,8,12,6,9$. The corresponding standardized labeling $w'$ is depicted in Figure \ref{fig:standardizedfamily}.
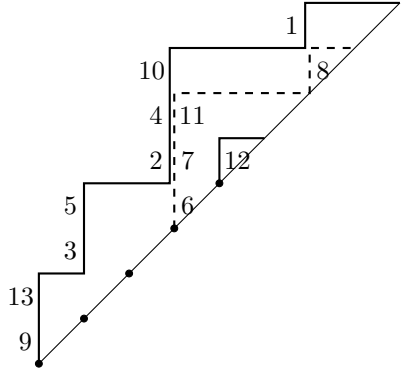
\begin{figure}[h]
    \centering
    \begin{tikzpicture}[scale=0.6]

        \draw (0,0)--(8,8);
        \draw[fill=black] (0,0) circle (.5ex);
        \draw[fill=black] (1,1) circle (.5ex);
        \draw[fill=black] (2,2) circle (.5ex);
        \draw[fill=black] (3,3) circle (.5ex);
        \draw[fill=black] (4,4) circle (.5ex);
        
        \draw[thick] (0,0)--(0,2)--(1,2)--(1,4)--(2.9,4)--(2.9,7)--(5.9,7)--(5.9,8)--(8,8);
        \draw[thick,dashed] (3,3)--(3,6)--(6,6)--(6,7)--(7,7);
        \draw[thick] (4,4)--(4,5)--(5,5);
    
        \node at (-.3,.5) {9};
        \node at (-.4,1.5) {13};
        \node at (.7,2.5) {3};
        \node at (.7,3.5) {5};
        \node at (2.6,4.5) {2};
        \node at (2.6,5.5) {4};
        \node at (2.5,6.5) {10};
        \node at (5.6,7.5) {1};
    
        \node at (3.3,3.5) {6};
        \node at (3.3,4.5) {7};
        \node at (3.4,5.5) {11};
        \node at (6.3,6.5) {8};
    
        \node at (4.4,4.5) {12};
    \end{tikzpicture}
    \caption{The standardization of the labeling in Figure \ref{fig:labeledfamily}.}
    \label{fig:standardizedfamily}
\end{figure}

\begin{lemma}\label{lem:standardization}
    For any labeling $w\in L(\pi)$, there exists a standard labeling $w' \in SL(\pi)$ such that the standardization of $\Read(w)$ is $\Read(w')$. Moreover, $\dinv(\pi,w)=\dinv(\pi,w')$.
\end{lemma}
\begin{proof}
    Let $w'$ be the assignment of labels to the north steps of $\pi$ from standardization $\Read'(w)$ of $\Read(w)$ using the same reading order. One easily checks that $w'$ satisfies the conditions to be a labeling of $\pi$ from Section \ref{sec:prelimLW}. Indeed, standardization preserves strict inequalities of labels, and in the notation of the second case when $r'=r$, $r'$ comes before $r$ in $\Read(w)$ so standardized pair of labels satisfies the required inequality in this case as well.

    To verify that $\dinv(\pi,w')=\dinv(\pi,w)$, we observe that any dinv pair for the labeling $w$ is still a dinv pair for $w'$, since the strict inequalities and relative positions between labels are maintained. It remains to check that no new dinv pairs are created from pairs of steps with equal labels in $w$ but distinct labels in $w'$. 
    
    Consider two north steps $N_1$ and $N_2$ with the same label $L_1 = L_2$ in $(\pi,w)$ but with distinct labels $L'_1<L'_2$ in $(\pi,w')$, respectively. By the construction of the standardization, $L'_2$ must be after $L'_1$ in $\Read(w')$. But dinv pairs always have the larger label on the first step in the reading order (see Figure \ref{fig:dinvpairs}), thus no new dinv pairs are created. 
\end{proof}

We recall the following definitions from \cite[Pages 91 and 99]{H08}. For a partition $\mu$ of $n$, a permutation $\sigma\in S_n$ is called a \emph{$\mu$-shuffle} if all of the sequences $[1,2,\dots,\mu_1],[\mu_1+1,\mu_1+2,\dots,\mu_1+\mu_2],\dots$ appear as ordered (but not necessarily consecutive) subsequences of $\sigma(1),\sigma(2),\dots,\sigma(n)$. For a pair of partitions $\mu,\eta$ with $|\mu|+|\eta|=n$, a permutation $\sigma\in S_n$ is called a \emph{$\mu,\eta$-shuffle} if all of the decreasing sequences $[\eta_1,\dots,1], [\eta_1+\eta_2,\dots,\eta_1+1],\dots$ and increasing sequences $[|\eta|+1,\dots,|\eta|+\mu_1],[|\eta|+\mu_1+1,\dots,|\eta|+\mu_1+\mu_2],\dots$ appear as ordered (but not necessarily consecutive) subsequences of $\sigma(1),\sigma(2),\dots,\sigma(n)$. 

\begin{lemma}\label{lem:Fpi}
    For any weakly-nested $\lambda$-family Dyck paths $\pi$ and partition $\mu$ with $|\mu|=|\lambda|$, we have
    \begin{equation}\label{eq.lemma3.2}
        \langle F_{\pi},h_\mu \rangle = \sum_{\substack{w\in SL(\pi)\\\Read(w) \text{ a }\mu\text{-shuffle}}}t^{\dinv(\pi,w)}.
    \end{equation}
\end{lemma}
\begin{proof}
As mentioned at the beginning of Section \ref{sec:Mac}, $F_{\pi}$ is symmetric in $\{x_i\}_{i \geq 1}$. Recall that $\{m_\mu\}$ and $\{h_\mu\}$ are dual bases of symmetric functions, it turns out that $\langle F_\pi,h_\mu\rangle$ is the coefficient on $m_\mu$ in the monomial expansion of $F_\pi$. Each monomial can be extracted from $F_\pi = \sum_{w\in L(\pi)} x^w t^{\dinv(\pi,w)}$ by fixing the weight of the labeling, giving the expression
\begin{equation}\label{eq.pflemma3.2}
    \langle F_\pi,h_\mu\rangle = \sum_{\substack{w\in L(\pi)\\ \wt(w) = \mu}} t^{\dinv(\pi,w)}.
\end{equation}
Then \eqref{eq.lemma3.2} follows from Lemma \ref{lem:standardization} that standardization provides a dinv-preserving bijection between labelings of $\pi$ of weight $\mu$ and standard labelings of $\pi$ whose reading word is a $\mu$-shuffle.
\end{proof}

Now, we are ready to prove Theorem \ref{thm:MacdonaldIntro}.
\begin{proof}[Proof of Theorem \ref{thm:MacdonaldIntro}]
    By ``superization'' \cite[Theorem 6.10]{H08} and Lemma \ref{lem:Fpi}, for any two partitions $\mu,\eta$ with $|\mu|+|\eta|=|\lambda|$, we have
    \begin{equation}\label{eq.pfthm1.1-1}
        \langle F_{\pi},h_\mu e_\eta  \rangle = \sum_{\substack{w\in SL(\pi)\\\Read(w) \text{ a }\mu,\eta \text{-shuffle}}}t^{\dinv(\pi,w)}.
    \end{equation}
    
    Taking $\mu= (d)$ and $\eta = (n-d)$, and applying linearity of the Hall inner product to the Loehr--Warrington formula $\nabla(s_\lambda) = \sgn(\lambda)\sum_\pi q^{\area(\pi)}F_\pi$ (Theorem \ref{thm:LW}), we therefore obtain 
    \begin{equation}\label{eq.pfthm1.1-2}
        \langle\nabla(s_\lambda), h_de_{n-d}\rangle = \sgn(\lambda)\sum_\pi q^{\area(\pi)}\langle F_\pi,h_d e_{n-d}\rangle = \sgn(\lambda)\hspace{-3ex}\sum_{\substack{(\pi,w)\in \SLNDP_\lambda\\\Read(w) \text{ a }(d),(n-d)\text{-shuffle}}}\hspace{-3ex}q^{\area(\pi)}t^{\dinv(\pi,w)}.
    \end{equation}

    The final sum of \eqref{eq.pfthm1.1-2} is over weakly-nested $\lambda$-families of Dyck paths $\pi$ with standard labeling $w$ such that $\Read(w)$ is a shuffle of the decreasing sequence $n-d,\dots,1$ and the increasing sequence $n-d+1,\dots,n$. We complete the proof by showing that such a pair $(\pi,w)$ can be interpreted as a strictly-nested $\lambda$-family of Schr\"oder paths with $d$ diagonal steps and that the area and dinv statistics of $(\pi,w)$ coincide with the statistics on a $\lambda$-family of Schr\"oder paths defined in Section \ref{sec:statistics}.

    First, we observe that for $(\pi,w)$ as above, no two paths in $\pi$ can share a north step. Indeed, suppose that paths $\pi_i$ and $\pi_j$ with $i<j$ share a north step. Since the paths cannot share east steps and cannot intersect at their starting points, the shared north step on $\pi_i$ is followed by another north step, and the shared north step on $\pi_j$ is preceded by a north step. Call the labels on these four paths $a,b,c$, and $d$ in reading order (as shown in Figure \ref{fig.shareNsteps}). By the conditions on labelings in Section \ref{sec:prelimLW} and the fact that all the labels are distinct since $w$ is standard, we have inequalities $b>a>c$ and $b>d>c$. Subsequences $a,b,c,d$ in $\Read(w)$ with these relative order conditions cannot appear in a $(d),(n-d)$-shuffle, since $b>a>c$ implies that $b$ is a part of the increasing sequence $n-d+1,\dots,n$ and $c$ is a part of the decreasing sequence $n-d,\dots,1$, but then there is no way to have $b>d>c$. 
    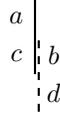
\begin{figure}[h]
        \centering
        \begin{tikzpicture}[scale=0.5]
    
            \draw[thick] (0,0)--(0,2);
            \draw[thick,dashed] (.1,-1)--(.1,1);
        
            \node at (-.5,.5) {$c$};
            \node at (-.5,1.5) {$a$};
            \node at (.5,.5) {$b$};
            \node at (.5,-.5) {$d$};
        \end{tikzpicture}
        \caption{An illustration of shared north steps and their labelings in the proof of Theorem \ref{thm:MacdonaldIntro}.}
        \label{fig.shareNsteps}
    \end{figure}

    Next, we observe that since the labels decrease down each column on any individual path, the sequence $n-d+1,\dots,n$ appearing in increasing order in $\Read(w)$ forces any north step with one of these labels to be immediately followed by an east step. We may therefore interpret $(\pi,w)$ as a family of Schr\"oder paths by replacing each north step with one of the labels $n-d+1,\dots,n$ and its following east step with a single diagonal step. 

    The only possible intersections between paths $\pi_i$ and $\pi_j$ with $i<j$ is a north step of $\pi_i$ that starts at the end of a north step of $\pi_j$. Call the labels on these paths $a$ and $b$ respectively, and note that $b>a$. Since $b$ is after $a$ in $\Read(w)$, $b$ must be a label from the increasing sequence $n-d+1,\dots,n$. Therefore, when we interpret the step labeled $b$ and its following east step as a diagonal step, it no longer intersects the path $\pi_i$. This shows that the resulting family of Schr\"oder paths is strictly nested. One easily checks that all strictly-nested families of Schr\"oder paths are obtained in this way.

    Finally, it is easy to see that the way we obtained a strictly-nested family of Schr\"oder paths mentioned above does not change the area statistic. We also observe that the dinv pairs of vertical and diagonal steps for $\lambda$-families of Schr\"oder paths defined in Section \ref{sec:statistics} coincide with the dinv pairs of labeled weakly-nested Dyck paths whose reading word is a $(d),(n-d)$-shuffle. This completes the proof of Theorem \ref{thm:MacdonaldIntro}.
\end{proof}

The key point of proving Corollary \ref{cor:symmetry} is to show that $\nabla(s_{\lambda})$ is symmetric in $q$ and $t$ for any partition $\lambda$. To this end, we state and prove below a more general result \cite{Loehr}, which is standard in the study of $q,t$-Catalan combinatorics and Macdonald polynomials.
\begin{theorem}[\cite{Loehr}]\label{thm:qtsym}
    Suppose $G$ is a symmetric function in $\Lambda$ and symmetric in $q$ and $t$. Then $\nabla(G)$ is also symmetric in $q$ and $t$.
\end{theorem}
\begin{proof}
    The Schur expansion of the modified Macdonald polynomials is given by
    \begin{equation}\label{eq.sHexp}
        \tilde{H}_{\mu}(\mathbf{x};q,t) = \sum_{\lambda \vdash n} \tilde{K}_{\lambda,\mu}(q,t) s_{\lambda}(\mathbf{x}),
    \end{equation}
    where $\tilde{K}_{\lambda,\mu}(q,t)$ is the modified $q,t$-Kostka polynomial (see \cite[Chapter 2]{H08}). Using the fact that $\tilde{K}_{\lambda,\mu}(q,t) = \tilde{K}_{\lambda,\mu^{\prime}}(t,q)$ where $\mu^{\prime}$ denotes the conjugate of $\mu$ (\cite[Equation (2.30)]{H08}), one can readily see that
    \begin{equation}\label{eq.Hconj}
        \tilde{H}_{\mu}(\mathbf{x};q,t) = \tilde{H}_{\mu^{\prime}}(\mathbf{x};t,q).
    \end{equation}

    Expressing $G$ in terms of the modified Macdonald polynomials yields the following: 
    \begin{align}
        G &= \sum_{\mu \vdash n} a_{\mu}(q,t)\tilde{H}_{\mu}(\mathbf{x};q,t) \label{eq.G-exp1}\\
         & = \sum_{\mu \vdash n} a_{\mu^{\prime}}(q,t)\tilde{H}_{\mu^{\prime}}(\mathbf{x};q,t) \label{eq.G-exp2} \\
         & = \sum_{\mu \vdash n} a_{\mu^{\prime}}(t,q)\tilde{H}_{\mu^{\prime}}(\mathbf{x};t,q) \label{eq.G-exp3},
    \end{align}
    for some $a_{\mu}(q,t) \in \mathbb{Q}(q,t)$. \eqref{eq.G-exp2} is obtained from \eqref{eq.G-exp1} by taking the conjugate of each $\mu$ and \eqref{eq.G-exp3} follows from the assumption that $G$ is symmetric in $q$ and $t$. By \eqref{eq.Hconj}, we have
    \begin{equation}\label{eq.G-exp4}
        a_{\mu}(q,t) = a_{\mu^{\prime}}(t,q).
    \end{equation}

    Applying the nabla operator to \eqref{eq.G-exp1} and using \eqref{eq.nabladef}, we obtain
    \begin{align}
        \nabla(G) & = \sum_{\mu \vdash n} a_{\mu}(q,t) \nabla \left( \tilde{H}_{\mu}(\mathbf{x};q,t) \right) \nonumber\\
         & = \sum_{\mu \vdash n} a_{\mu}(q,t) q^{n(\mu^{\prime})} t^{n(\mu)} \tilde{H}_{\mu}(\mathbf{x};q,t). \label{eq.G-exp5}
    \end{align}
    Taking the conjugate of each $\mu$ in \eqref{eq.G-exp5}, by \eqref{eq.Hconj} and \eqref{eq.G-exp4}, this shows that $\nabla(G)$ is symmetric in $q$ and $t$.
\end{proof}
\begin{proof}[Proof of Corollary \ref{cor:symmetry}]
    Recall from \eqref{eq.P-polynomial}, $P_{\lambda}(z;q,t)$ is the generating polynomial of domino tilings of $R_{\lambda}$, equivalently, the weighted sum of $\lambda$-families of Schr\"oder paths. By Theorem \ref{thm:MacdonaldIntro}, we have the following expression.
    \begin{equation}
        P_{\lambda}(z;q,t) = \sgn(\lambda) \sum_{d=0}^{n} \langle \nabla(s_{\lambda}), h_d e_{n-d} \rangle z^d.
    \end{equation}
    
    By Theorem \ref{thm:qtsym}, $\nabla(s_{\lambda})$ is symmetric in $q$ and $t$ for any partition $\lambda$. This implies that $\langle \nabla(s_{\lambda}), h_d e_{n-d} \rangle$ is symmetric in $q$ and $t$, so is the polynomial $P_{\lambda}(z;q,t)$.
\end{proof}


\section{Partitions of Square Shape}\label{sec:pfAD}

For the rest of the paper, we continue with the same notations introduced in Section \ref{sec:prelim}. This section focuses on the special case when $\lambda = (n^n)$ (the Ferrers diagram is a square). In Section \ref{sec:exADASM}, we explain how the region $R_{(n^n)}$ relates to $\AD_n$ and provide an example when $n=2$. We then show how the area (Section \ref{sec:shuffle-area}) and dinv (Section \ref{sec:shuffle-dinv}) statistics of a domino tiling of $R_{(n^n)}$ change under the domino shuffling. Finally, the proof of \cref{thm:introAD} is presented in Section \ref{sec:pfADformula}.

\subsection{Extended Aztec Diamonds and Extended ASMs}\label{sec:exADASM}

Decomposing the partition $\lambda = (n^n)$ into border strips gives $(n_0,n_1,\dots,n_k) = (2k+1,2k-1,\dots,3,1)$, where $k=n-1$. Recall that in the top half of $\AD_{2n}$, the boxes in the bottom row are labeled as $0,0,1,1,2,2,\dots,n-1,n-1$ from left to right. The region $R_{(n^n)}$ is formed by removing specific boxes from the bottom row of the top half of $\AD_{2n}$: the white boxes labeled $0,1,\dots,n-1$ and the black boxes labeled $2n-1,2n-2,\dots,n$. In any domino tiling of $R_{(n^n)}$, vertical dominoes are forced to cover certain parts of the region. After placing these forced vertical dominoes, the remaining region reduces to $\AD_n$. Figure \ref{fig:square} illustrates the border strip decomposition of $\lambda=(4^4)$ and Figure \ref{fig:reduce} shows the region $R_{(4^4)}$, the subregion on the two sides of $R_{(4^4)}$ is covered by vertical dominoes. For convenience, the region $R_{(n^n)}$ is referred to as the \emph{extended $\AD_n$}.

Given a domino tiling $T$ of the extended $\AD_n$, one can obtain its corresponding \emph{extended} ASM similarly as in Section \ref{sec:ASMAD}. The idea is to obtain the ASM $M_e(T|_{\AD_n})$ of size $n+1$ by restricting $T$ to the region $\AD_n$ and then we add a triangular array of $0$'s of size $n$ on the southwestern and southeastern sides, respectively. The resulting extended ASM $M_e(T)$ is a triangular array of size $2n+1$. Similarly, the extended ASM $M_o(T)$ can be obtained from $M_o(T|_{\AD_n})$ of size $n$ by adding a triangular array of $0$'s of size $n-1$ on the southwestern and southeastern sides, respectively. Sometimes, it will be helpful to view the extended ASMs $M_e(T)$ and $M_o(T)$ in this way: we map the even (resp., odd) points within $\AD_n$ via $M_e$ (resp., $M_o$), for those even (resp., odd) points outside $\AD_n$, they are all mapped to be $0$. Figure \ref{fig:extASMAD2} provides a domino tiling $T$ of the extended $\AD_4$ and its extended ASM $M_o(T)$. Another example of a domino tiling $T$ of the extended $\AD_3$ and its extended ASM $M_e(T)$ is given in Figures \ref{fig:exTodd} and \ref{fig:extasm}.
\begin{figure}[h]
    \centering
    \subfigure[]{\label{fig:square}\includegraphics[height=0.16\textwidth]{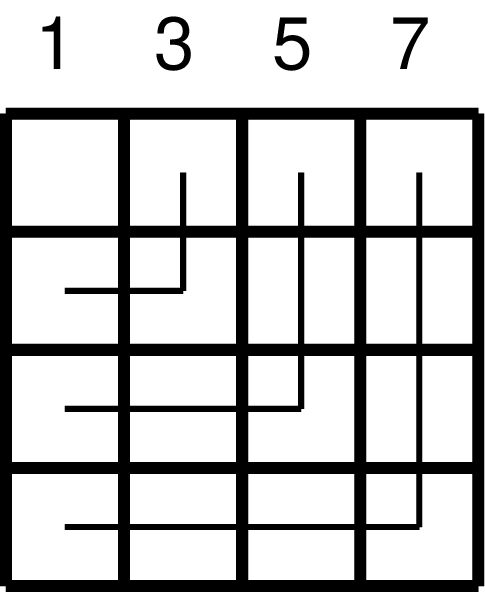}}
    \hspace{5mm}
    \subfigure[]{\label{fig:reduce}\includegraphics[height=0.21\textwidth]{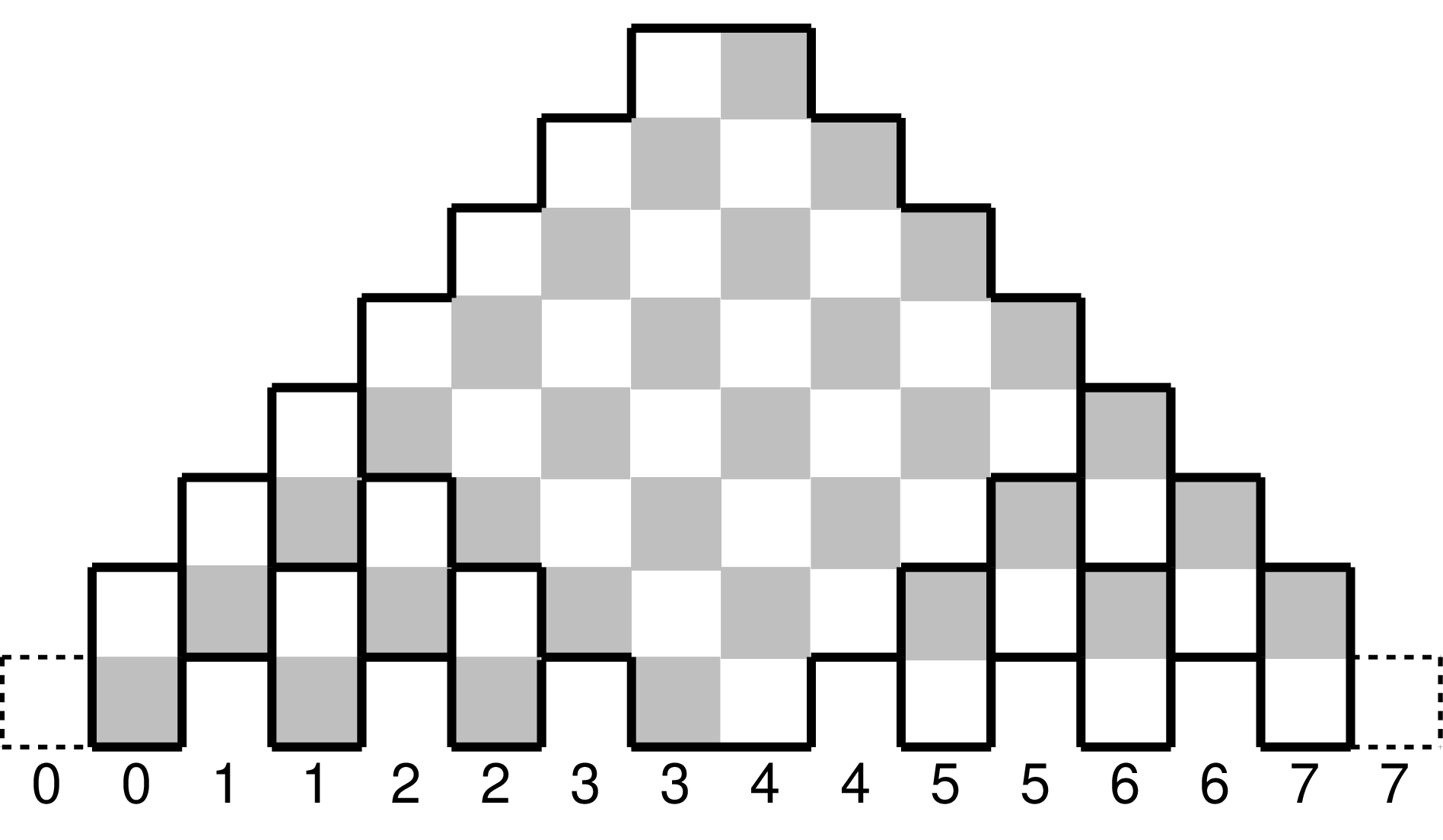}}
    \hspace{5mm}
    \subfigure[]{\label{fig:extASMAD2}\includegraphics[height=0.2\textwidth]{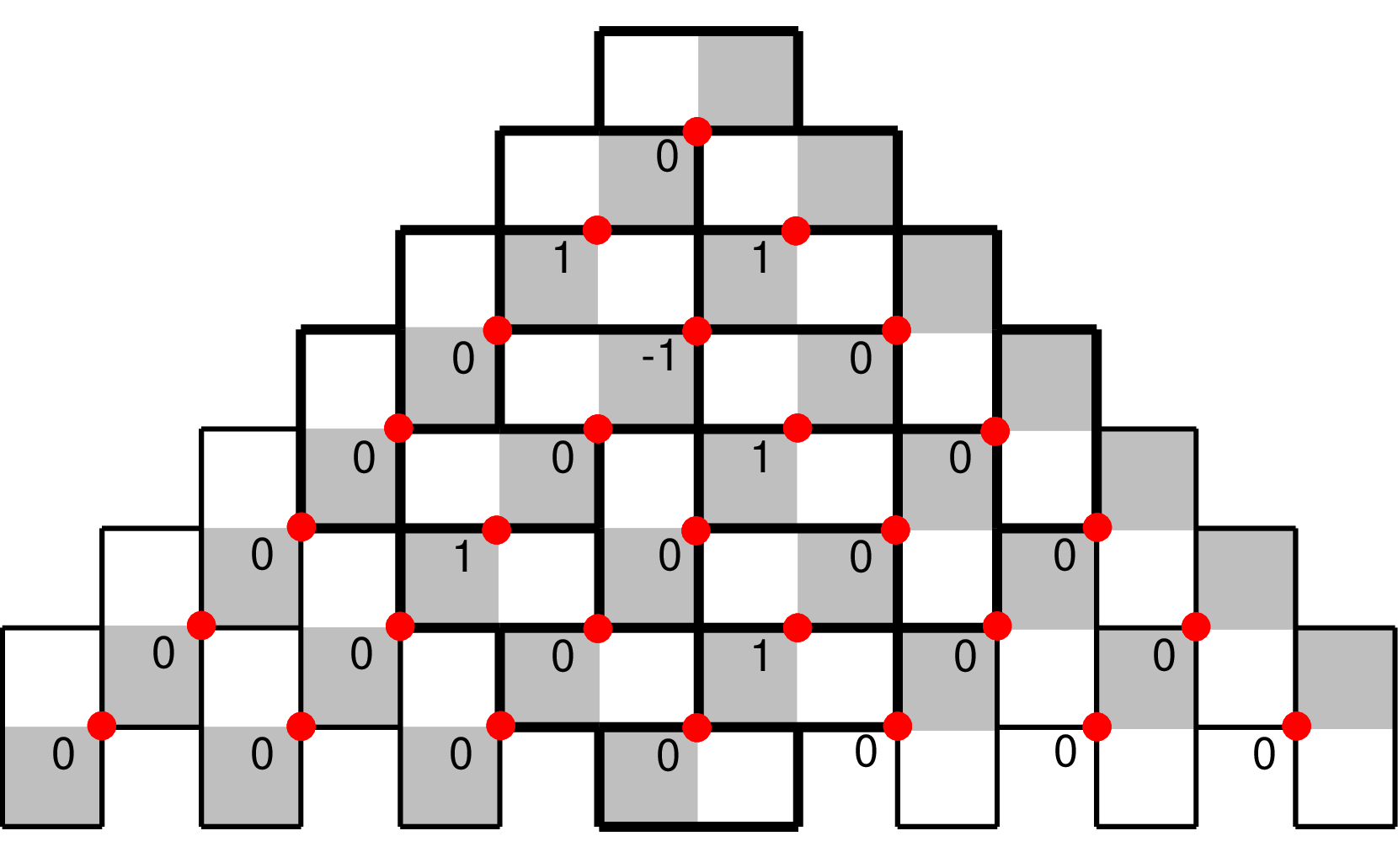}}
    \caption{(a) The decomposition of $\lambda=(4^4)$ into border strips of lengths $(n_0,n_1,n_2,n_3) = (7,5,3,1)$. (b) The region $R_{(4^4)}$. (c) A domino tiling $T$ of the extended $\AD_4$ and its extended ASM $M_o(T)$.}
\end{figure} 

We would like to point out that the dinv adjustment simplifies to $\adj(\lambda) = \binom{n}{2}$ when $\lambda = (n^n)$, so the dinv statistic reduces to
\begin{equation}\label{eq.def-domino-dinv-reduce}
    \dinv(T) = \binom{n}{2} + \left| \{ \text{domino pairs in } T \} \right|.
\end{equation}

\begin{example}
    In Table \ref{tab:AD2example}, we list $8$ weighted domino tilings of the extended $\AD_2$ where the original $\AD_2$ is enclosed by bold edges. One can check that their total weight factors into $\AD_2(z;q,t) = (qt)^2 (z+1)(z+q)(z+t)$, which agrees with \eqref{eq.ADproduct} when $n=2$.
\end{example}
\begin{table}[htb!]
    \centering
    \begin{tabular}{c|c|c|c}
            \includegraphics[width=0.16\linewidth]{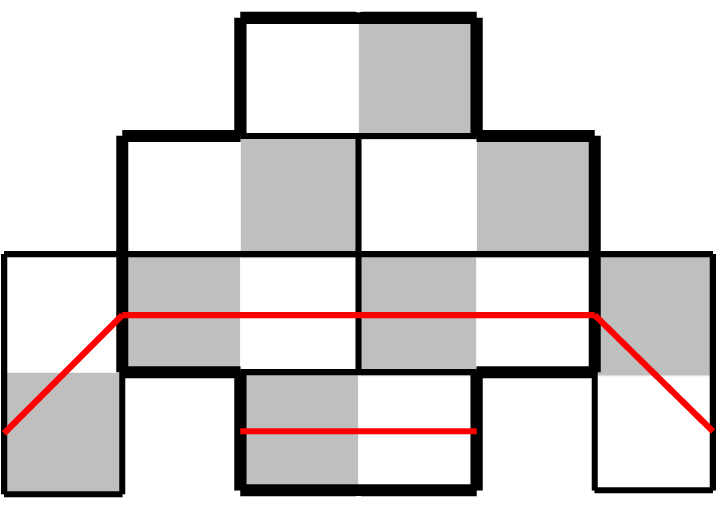}
            &  \includegraphics[width=0.16\linewidth]{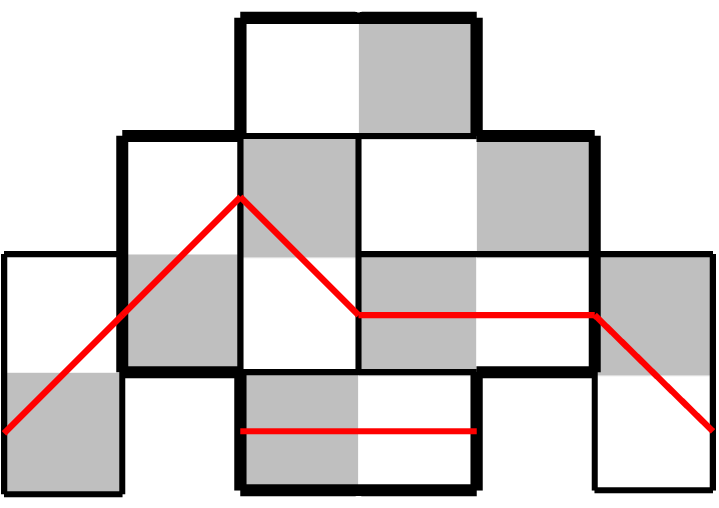} &  \includegraphics[width=0.16\linewidth]{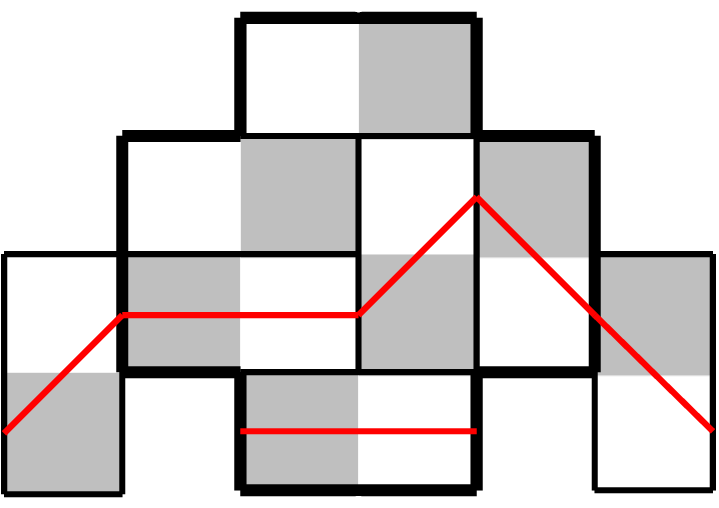} & \includegraphics[width=0.16\linewidth]{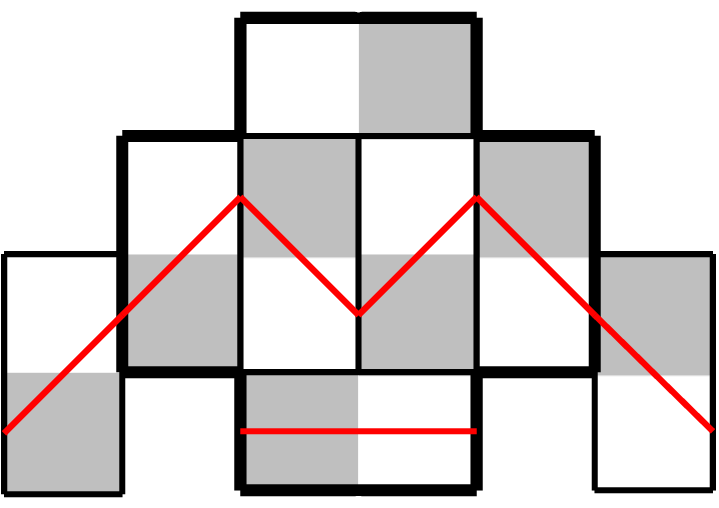} \\
        $z^3q^2t^2$ & $z^2q^2t^3$ & $z^2q^2t^2$ & $z^1q^2t^3$ \\
         \hline
        \includegraphics[width=0.16\linewidth]{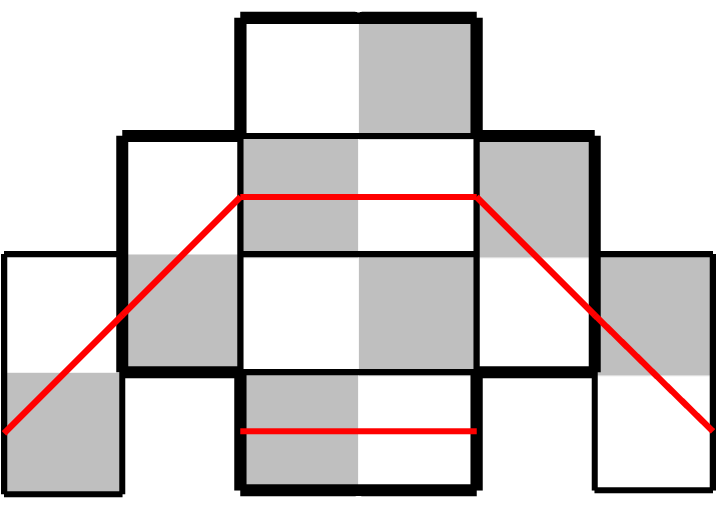} & \includegraphics[width=0.16\linewidth]{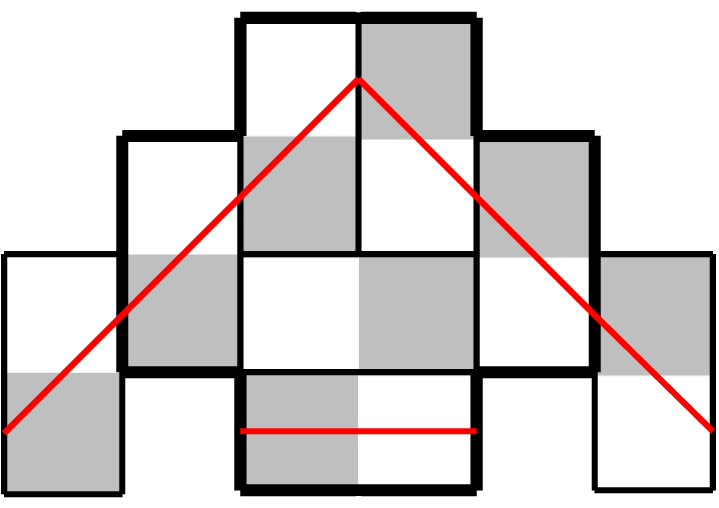} & \includegraphics[width=0.16\linewidth]{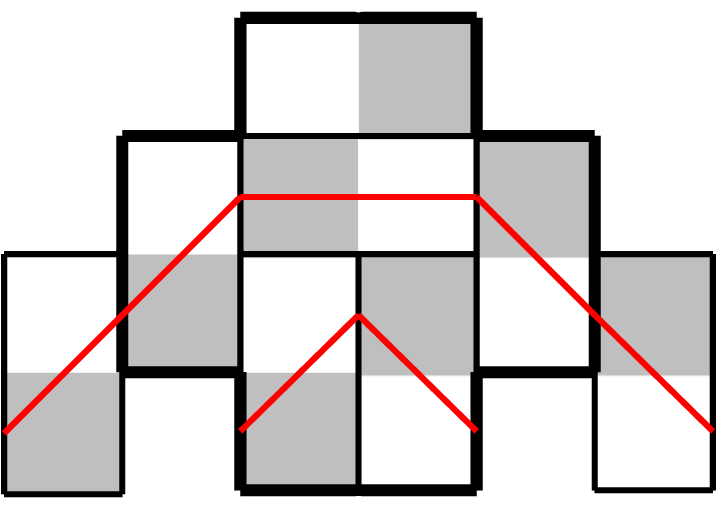} & \includegraphics[width=0.16\linewidth]{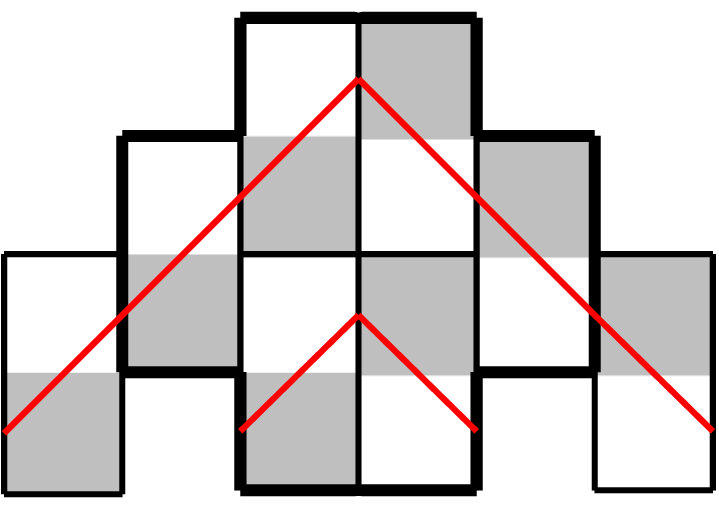} \\
        $z^2q^3t^2$ & $z^1q^3t^2$ & $z^1q^3t^3$ & $z^0q^3t^3$
    \end{tabular}
    \caption{The weight $z^{\diags(T)}q^{\area(T)}t^{\dinv(T)}$ of each domino tiling $T$ of the extended $\AD_2$.}
    \label{tab:AD2example}
\end{table}

\subsection{Domino shuffling and the area statistic}\label{sec:shuffle-area}

Given a domino tiling $T$ of the extended $\AD_n$ with $m$ odd blocks. Let $T_o$ be the odd-deficient tiling of $T$ (see Section \ref{sec:introshuffling}) and $\hat{T}_o$ be the tiling obtained from $T_o$ by completing all odd blocks with two \emph{horizontal} dominoes. We define the area and diags statistics on the odd-deficient tiling $T_o$ by
\begin{equation}\label{eq.def-area-To}
    \area(T_o) = \area(\hat{T}_o) \text{ and } \diags(T_o) = \diags(\hat{T}_o) - m.
\end{equation}

The domino shuffling introduced in Section \ref{sec:introshuffling} can be applied to an odd-deficient domino tiling $T_o$ of the extended $\AD_n$ similarly. We first attached $2n$ vertical dominoes at the bottom such that the bottom edges of these $2n$ vertical dominoes are aligned with the line $y=-1$; see the dominoes highlighted in blue in Figure \ref{fig:exTodd} for an example. Next, we apply the domino shuffling process to $T_o$ and these attached vertical dominoes, and then obtain the even-deficient tiling $T_e = S(T_o)$. Figure \ref{fig:exTeven} shows the resulting $T_e$ under the domino shuffling of $T_o$ given in Figure \ref{fig:exTodd}. Since those forced vertical dominoes on the southwestern (resp., southeastern) side of $\AD_n$ are moved one unit to their left (resp., right), creating a space of $\AD_{n+1}$ in the middle, this will not lead to conflicts under the domino shuffling process. Similarly, we set $\hat{T}_e$ to be the tiling obtained from $T_e$ by completing $m+n+1$ even blocks with two \emph{horizontal} dominoes. Define the area and diags statistics on the even-deficient tiling $T_e$ by
\begin{equation}\label{eq.def-area-Te}
    \area(T_e) = \area(\hat{T}_e) \text{ and } \diags(T_e) = \diags(\hat{T}_e) - (m+n+1).
\end{equation}

We show below how the area statistic changes under the domino shuffling.
\begin{proposition}\label{prop:area-shuffle}
    Let $T$ be a domino tiling of the extended $\AD_n$ with $m$ odd blocks, $T_o$ be the odd-deficient tiling of $T$, and $T_e = S(T_o)$ be the even-deficient tiling under the domino shuffling of $T_o$. Then
    \begin{equation}\label{eq:area-shuffle}
        \area(T_e) = \area(T_o) - \diags(T_o) + n(2n+1) - m.
    \end{equation}
\end{proposition}

Recall from \eqref{eq.def-domino-area} that the area of a domino tiling $T$ can be expressed as the sum of $\area(d_v)$ and $\area(d_h)$ for all vertical dominoes $d_v \in V(T)$ and horizontal dominoes $d_h \in H(T)$. To prove Proposition \ref{prop:area-shuffle}, we compare the area of dominoes in $V(\hat{T}_o) \cup H(\hat{T}_o)$ with the area of dominoes in $V(\hat{T}_e) \cup H(\hat{T}_e)$ under the domino shuffling. A central idea is to represent the area of the horizontal domino in each odd or even block as a statistic on the extended ASM $M_e(T)$. This statistic is defined as follows: we label the rows of $M_e(T)$ from bottom to top by $1$ to $2n+1$. If $x$ is an entry in $M_e(T)$ located at row $i$, then the \emph{level} of $x$ is defined as $\level(x) = i$. 

Assuming the bottom edge of the horizontal domino $d_h \in H(\hat{T}_o)$ in an odd block is on $y=i$, this implies that $\area(d_h) = i$. Also, the center of this odd block corresponds to the entry $x=-1$ in the extended ASM $M_e(T)$, where $\level(x) = i+1$. After the domino shuffling of $T_o$, assuming the bottom edge of the horizontal domino $d^{\prime}_h \in H(\hat{T}_e)$ in an even block is on $y=j$, that is, $\area(d^{\prime}_h) = j$. The center of this even block corresponds to the entry $y=1$ in the same extended ASM, where $\level(y) = j+1$. For example,
in Figure \ref{fig:exTodd}, if we fill the only odd block with two horizontal dominoes, let $d_h$ be one of these horizontal dominoes in $H(\hat{T}_o)$, then the bottom edge of $d_h$ is on $y=4$, which implies that $\area(d_h) = 4$. At the same time, this odd block corresponds to the entry $x=-1$ shown in Figure~\ref{fig:extasm} and $\level(x) = 5$. Similarly, in Figure \ref{fig:exTeven}, if we fill the the green-colored even block with two horizontal dominoes, let $d^{\prime}_h$ be one of these horizontal dominoes in $H(\hat{T}_e)$, then the bottom edge of $d^{\prime}_h$ lies on $y=3$, giving $\area(d^{\prime}_h) = 3$. This even block corresponds to the entry $x=1$ highlighted in the green box in Figure \ref{fig:extasm} and $\level(x) = 4$. 

Therefore, the area of these particular horizontal dominoes is precisely captured by the level statistic on $M_e(T)$. We summarize the above discussion in the following lemma.
\begin{lemma}\label{lem:area-level}
    Let $T$ be a domino tiling of the extended $AD_n$. Let $d_h \in H(\hat{T}_o)$ be a horizontal domino contained in an odd block of $\hat{T}_o$ and $d^{\prime}_h \in H(\hat{T}_e)$ be a horizontal domino contained in an even block of $\hat{T}_e$. Let $x$ be an entry of $M_e(T)$ corresponding to the center of an odd or even block. Then 
    \begin{equation}\label{eq.area-level}
       \level(x)-1 = \begin{cases}
         \area(d_h) & \text{if $x=-1$,} \\
         \area(d^{\prime}_h) & \text{if $x=1$.} 
          \end{cases}
    \end{equation}
\end{lemma}

\begin{figure}[hbt!]
    \centering
    \subfigure[]{\label{fig:exTodd}\includegraphics[height=0.155\textwidth]{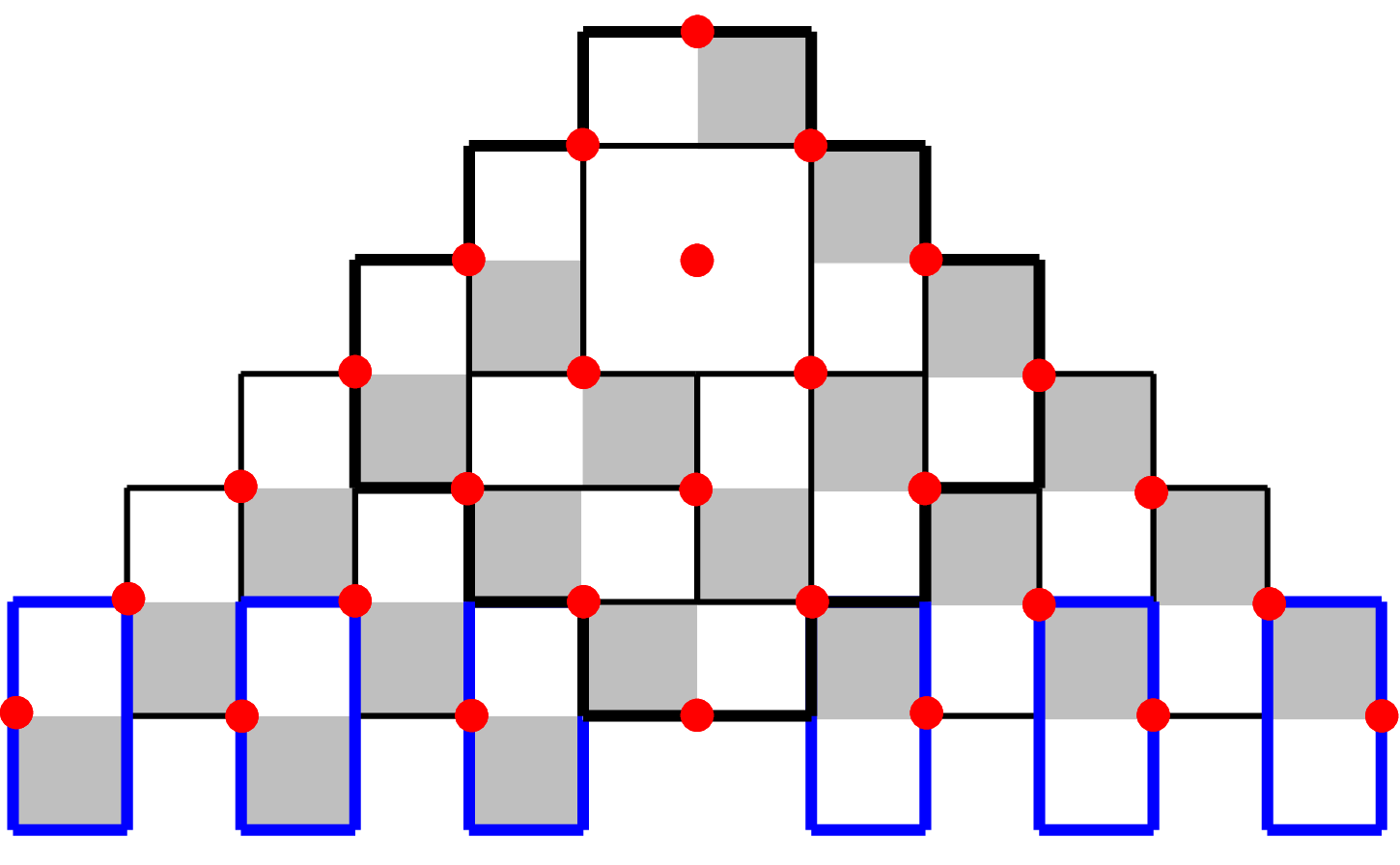}}
    \hspace{3mm}
    \subfigure[]{\label{fig:exTeven}\includegraphics[height=0.175\textwidth]{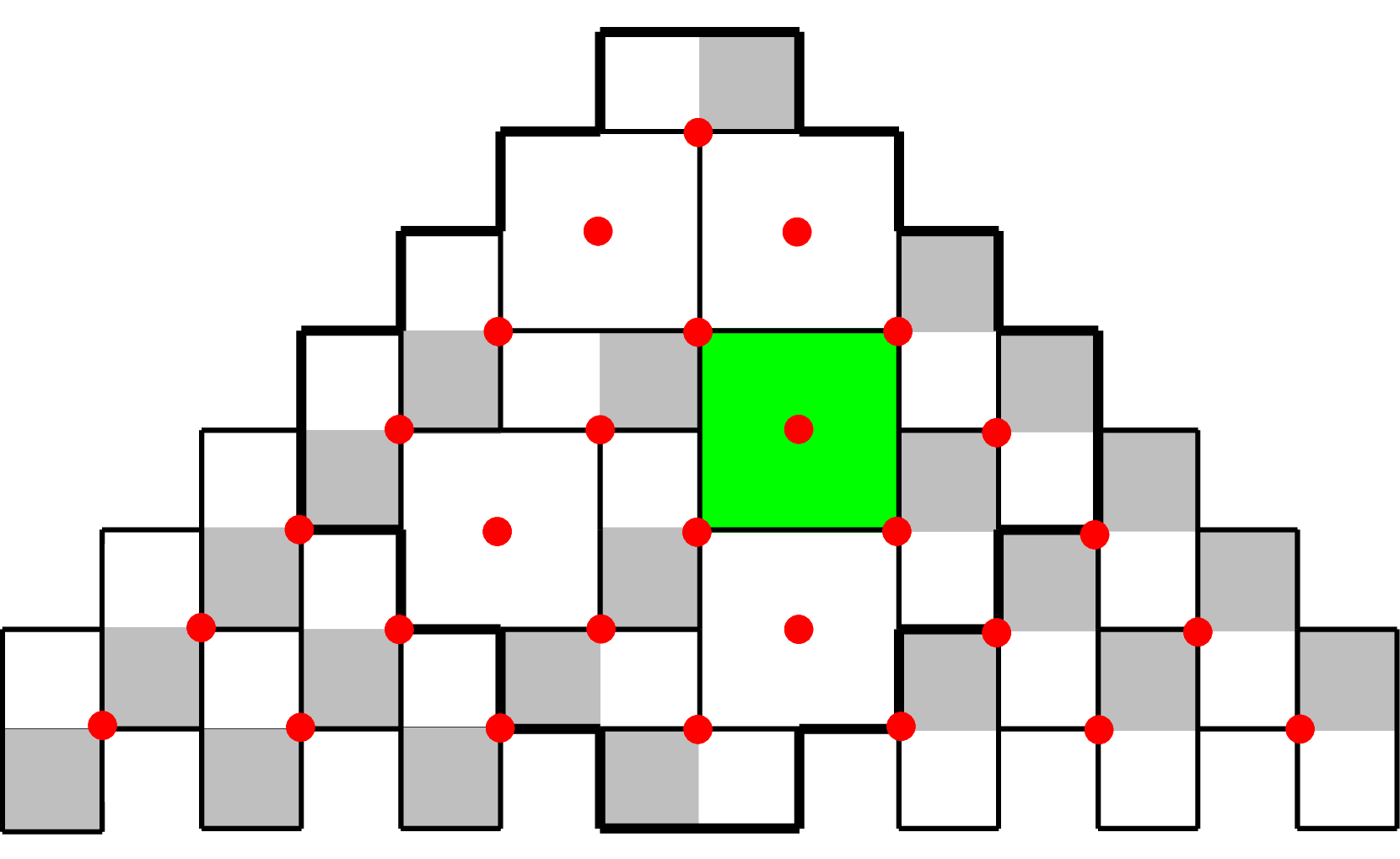}}
    \hspace{3mm}
    \subfigure[]{\label{fig:extasm}\includegraphics[height=0.16\textwidth]{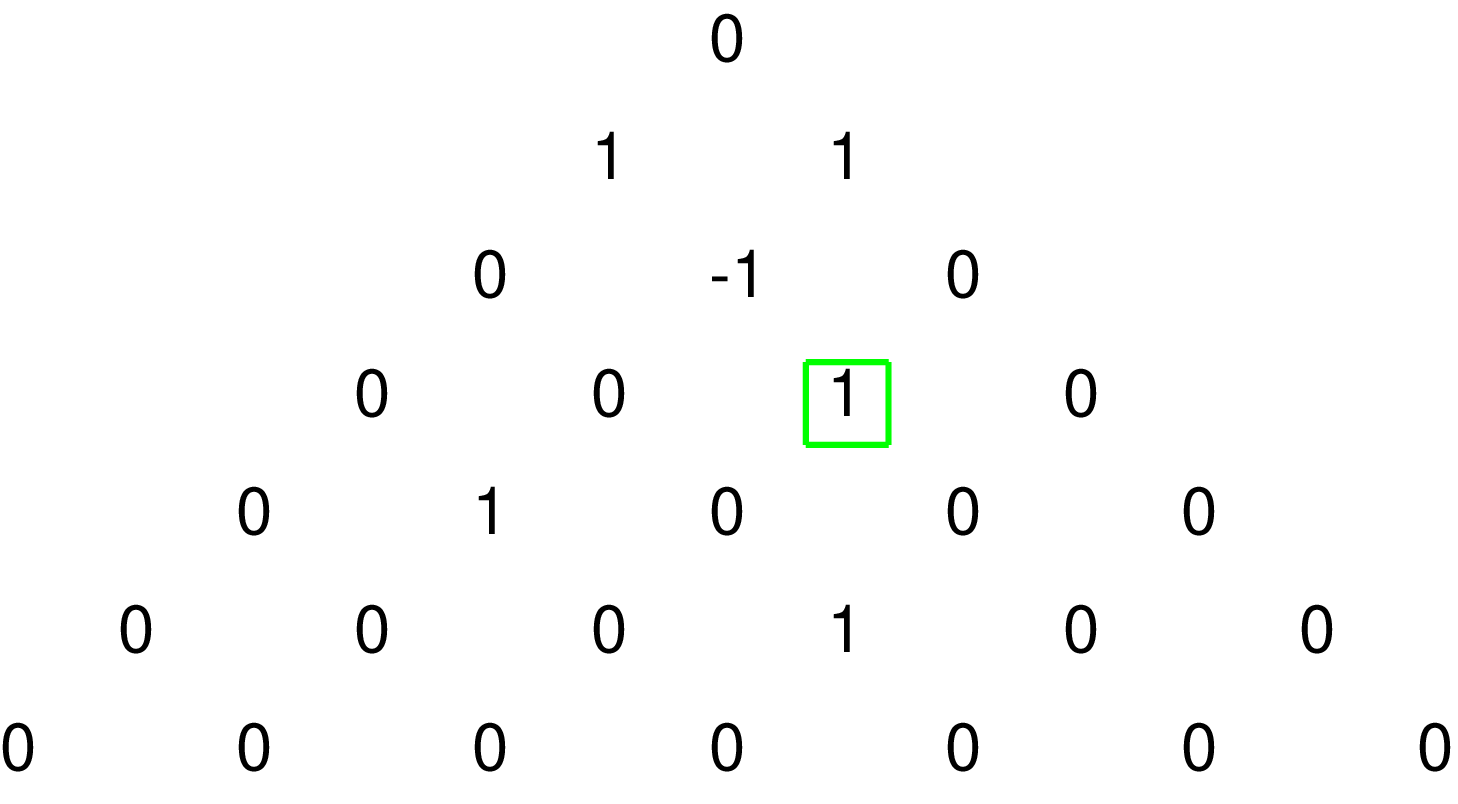}}
    \caption{(a) An odd-deficient tiling $T_o$ of the extended $\AD_3$, even points are marked in red and the attached vertical dominoes are highlighted in blue. (b) The even-deficient tiling $T_e$ under the domino shuffling of $T_o$ shown in Figure \ref{fig:exTodd}, odd points are marked in red. (c) The extended ASM $M_e(T_o) = M_o(T_e)$.}\label{fig:shuffleexample}
\end{figure}

Now, we are prepared to prove Proposition \ref{prop:area-shuffle}.

\begin{proof}[Proof of Proposition \ref{prop:area-shuffle}]
    Based on \eqref{eq.def-domino-area}, \eqref{eq.def-area-To} and \eqref{eq.def-area-Te}, we analyze the area of dominoes in the following two cases: (1) compare the area of dominoes that are \emph{not} in odd blocks of $\hat{T}_o$ with their area contributions after shuffling, and (2) compare the area contributions of odd blocks of $\hat{T}_o$ with the area contribution of even blocks of $\hat{T}_e$ after shuffling.

    \noindent Case (1). Let $d_h \in H(\hat{T}_o)$ be a horizontal domino that does not belong to any odd block of $\hat{T}_o$. We assume that the bottom edge of $d_h$ is on $y=i$, that is, $\area(d_h)=i$. During the domino shuffling process, this domino is shifted one unit downward. For simplicity, we adjust the $y$-coordinate by shifting it downward by one unit, aligning the bottom edge of the resulting even-deficient tiling $T_e$ with $y=0$. It turns out that the bottom edge of this domino is still on $y=i$ after shuffling, this implies that $\area(d_h)$ does not change after shuffling.
    
    On the other hand, let $d_v \in V(\hat{T}_o)$ be a vertical domino that does not belong to any odd block of $\hat{T}_o$. suppose the bottom edge of $d_v$ is on $y=i$. During the domino shuffling process, $d_v$ is shifted one unit to its left. Due to the coordinate adjustment described earlier, the bottom edge of this domino in $T_e$ is now located at $y=i+1$. Consequently, $d_v$ contributes one more to the area statistic on $T_e$ than to $T_o$. Therefore, in this case, the additional area contribution under the domino shuffling is equivalent to the number of vertical dominoes in $V(\hat{T}_o)$.
    
    Based on the basic counting, the number of dominoes in a domino tiling $A \in \T(\AD_n)$ is twice the sum of the number of vertical dominoes in $V(A)$ and the number of horizontal dominoes in $H(A)$. This quantity equals to $n(n+1)$. Thus,
    \begin{equation}\label{eq.ADcounting}
        |V(A)| = n(n+1)/2 - |H(A)| = n(n+1)/2 - \diags(A).
    \end{equation}
    To compute $|V(\hat{T}_o)|$, we restrict $\hat{T}_o$ to $\AD_n$. By \eqref{eq.def-area-To} and \eqref{eq.ADcounting},
    \begin{equation}\label{eq.area-up}
        |V(\hat{T}_o|_{\AD_n})| = n(n+1)/2-\diags(\hat{T}_o|_{\AD_n}) = n(n+1)/2 - \diags(T_o|_{\AD_n})-m.
    \end{equation}
    Additionally, there are $n(n-1)/2$ forced vertical dominoes in $V(T_o)$ within the extended $\AD_n$ that have not yet been accounted for. Therefore, together with \eqref{eq.area-up}, the additional area contribution under the domino shuffling in this case is given by
    \begin{equation}\label{eq.area-case1}
        |V(\hat{T}_o)| = |V(\hat{T}_o|_{\AD_n})| + n(n-1)/2 = n^2-\diags(T_o)-m.
    \end{equation}

    \noindent Case (2). Let $H_{\text{block}}(\hat{T}_o)$ (resp., $H_{\text{block}}(\hat{T}_e)$) be the subset of $H(\hat{T}_o)$ (resp., $H(\hat{T}_e)$) consisting of horizontal dominoes in odd blocks of $\hat{T}_o$ (resp., $\hat{T}_e$). By Lemma \ref{lem:area-level}, the area of a horizontal domino in $H_{\text{block}}(\hat{T}_o)$ or $H_{\text{block}}(\hat{T}_e)$ can be interpreted by the level statistic of $M_e(T)$. As a consequence, the total difference of the area contribution in this case can be written as
    \begin{equation}\label{eq.area-odd/even}
        \sum_{d^{\prime}_h \in H_{\text{block}}(\hat{T}_e)} \area(d^{\prime}_h) - \sum_{d_h \in H_{\text{block}}(\hat{T}_o)} \area(d_h) = \left( \sum_{\substack{x \in M_e(T) \\ x=1}} (\level(x)-1)\right) - \left( \sum_{\substack{x \in M_e(T) \\ x=-1}} (\level(x)-1) \right).
    \end{equation}

    It suffices to consider the level statistic on the entries of the original ASM $M_e(T|_{\AD_n})$ of size $n+1$ (since other entries are $0$). From bottom to top, we now label the arrays of $M_e(T|_{\AD_n})$ parallel to $y=x$ by $i=1$ to $i=n+1$ and the arrays parallel to $y=-x$ by $j=1$ to $j=n+1$. It is clear that the level of the entry $a_{i,j}$ of $M_e(T|_{\AD_n})$ is given by
    \begin{equation}\label{eq.levelcoordinate}
        \level(a_{i,j})=i+j-1.
    \end{equation}

    The key step to evaluate \eqref{eq.area-odd/even} is to identify it as a sum of $\level(a_{i,j})a_{i,j}$ over all the entries $a_{i,j}$ of $M_e(T|_{\AD_n})$ and then simplify it using \eqref{eq.levelcoordinate}. Then \eqref{eq.area-odd/even} can be rewritten as
    \begin{align}
    \sum_{i,j=1}^{n+1}(\level(a_{i,j})-1)a_{i,j} & = \sum_{i,j=1}^{n+1}(i+j-2)a_{i,j} \nonumber \\
     & = \sum_{i,j=1}^{n+1}i a_{i,j} + \sum_{i,j=1}^{n+1}j a_{i,j} - 2\sum_{i,j=1}^{n+1}a_{i,j} \nonumber \\
     & = \sum_{i=1}^{n+1}i \sum_{j=1}^{n+1}a_{i,j} + \sum_{j=1}^{n+1}j \sum_{i=1}^{n+1} a_{i,j} - 2\sum_{i,j=1}^{n+1}a_{i,j}. \label{eq.area-simplify}
    \end{align}        
    Using the fact that each row sum and column sum of an ASM is $1$, then \eqref{eq.area-simplify} reduces to 
    \begin{equation}\label{eq.area-even/odd-contribution}
        \frac{(n+2)(n+1)}{2} + \frac{(n+2)(n+1)}{2} - 2 (n+1) = n(n+1).
    \end{equation}

    Finally, combining the discussion from these two cases, by \eqref{eq.area-case1} and \eqref{eq.area-even/odd-contribution},
    \begin{align}
        \area(T_e) - \area(T_o) & = \area(\hat{T}_e) - \area(\hat{T}_o) \nonumber \\
        & = n^2 - \diags(T_o) - m + n(n+1) \nonumber \\
        & = n(2n+1) - \diags(T_o) - m. \nonumber
    \end{align}
    This completes the proof of Proposition \ref{prop:area-shuffle}.
\end{proof}

\subsection{Domino shuffling and the dinv statistic}\label{sec:shuffle-dinv}

Given a domino tiling $T$ of the extended $\AD_n$. Let $T_o$, $\hat{T}_o$, $T_e$, and $\hat{T}_e$ be as in the beginning of Section \ref{sec:shuffle-area}. The dinv statistic on an odd-deficient tiling $T_o$ and even-deficient tiling $T_e$ are defined as
\begin{equation}\label{eq.def-dinv-odd/even}
    \dinv(T_o) = \dinv(\hat{T}_o)  \text{ and } \dinv(T_e) = \dinv(\hat{T}_e).
\end{equation}
We show below how the dinv statistic changes under the domino shuffling.
\begin{proposition}\label{prop.dinv-shuffle}
    Given a domino tiling $T$ of the extended $\AD_n$, let $T_o$ be the odd-deficient tiling of $T$ and $T_e=S(T_o)$ be the even-deficient tiling under the domino shuffling of $T_o$. Then
    \begin{equation}\label{eq.dinv-shuffle}
        \dinv(T_e) = \dinv(T_o) + n^2 + \binom{n+1}{2}.
    \end{equation}
\end{proposition}

As noted in \eqref{eq.dominopairs}, the number of domino pairs in $T$ can be expressed in two different ways as the sum of the dinv statistic on vertical dominoes in $V(T)$, depending on the choice of the reference domino in the domino pairs. The main idea of proving Proposition \ref{prop.dinv-shuffle} is to express the dinv of each vertical domino in $V(\hat{T}_o)$ or $V(\hat{T}_e)$ as a statistic on the extended ASM $M_e(T)$. We introduce the following two statistics on $M_e(T)$. Label the rows of $M_e(T)$ from bottom to top by $1$ to $2n+1$ and let $x$ be an entry of $M_e(T)$ at row $i$. The first statistic is called the \emph{northeastern sum} of $x$, denoted by $\mathsf{ne}(x)$, and is defined to be the sum of entries of $M_e(T)$ located on the northeastern side of $x$ (does not include $x$ itself) along a line parallel to $y=x$. For example, the second entry $x=0$ counted from the left at row $1$ in Figure \ref{fig:extasm} has $\mathsf{ne}(x) = 0+1+0+(-1)+1 = 1$. 

The second statistic is called the \emph{top partial sum} of $x$, denoted by $\tp(x)$, and is defined to be the sum of all entries of $M_e(T)$ located from the rows $i+1$ to $2n+1$ and to the right of $x$ on the same row. For example, let $x$ be the entry in the green box in Figure \ref{fig:extasm}, and note that $x$ is located at row $4$. Then $\tp(x)$ is the sum of entries located from the rows $5$ to $7$ and the entries to the right of $x$, that is, $\tp(x) = 0+1+1+0+(-1)+0+0 = 1$.

Observe that in any domino tiling $T$, one of the midpoints of the two longer edges of a domino must be an odd point, while the other midpoint must be an even point. This implies that one of the two midpoints is mapped to an entry of the extended ASM under $M_e$ or $M_o$. An entry $x$ of the extended ASM $M_e(T)$ (or $M_o(T)$) is said to be \emph{associated} with a domino in $T$ if $x = M_e(v)$ (or $M_o(v)$), where $v$ is an even (or odd) point and is the midpoint of the longer edge of that domino. We have seen that at the beginning of the paragraph, for each domino in $T$, there exists an entry of $M_e(T)$ (or $M_o(T)$) that is associated with it. We present a result of the statistic $\mathsf{ne}(x)$ in the following lemma, the interested reader can verify it from the example given in Figure \ref{fig:shuffleexample}.
\begin{lemma}\label{lem.nesum}
   Let $T$ be a domino tiling of the extended $\AD_n$ and $x$ be an entry of its extended ASM $M_e(T)$ or $M_o(T)$. Then
   \begin{equation}\label{eq.lemmanesum}
       \mathsf{ne}(x) = \begin{cases}
            0, & \text{ if $x=1$, } \\
            1, & \text{ if $x=-1$, } \\
            1, & \text{ if $x=0$ and $x$ is associated with a domino in $V(T) \cup H(T)$. } \\
            0, & \text{ if $x=0$ and $x$ is associated with a domino in $V^c(T) \cup H^c(T)$. }
       \end{cases}
   \end{equation}
\end{lemma}
\begin{proof}
    Let $v_1,v_2,\dots,v_d$ be the entries of the extended ASM located on the northeastern side of $x$ along a line parallel to $y=x$. Note that $\mathsf{ne}(x) = \sum_{i=1}^{d}v_i$. According to the definition of the ASM (Section \ref{sec:ASMAD}), it is clear that if $x=1$, then $\mathsf{ne}(x)=0$, and if $x=-1$, then $\mathsf{ne}(x) = 1$. When $x=0$, we prove the lemma for the cases when $x$ is associated with a vertical domino in $V(T)$ or in $V^c(T)$; the cases when $x$ is associated with a horizontal domino in $H(T)$ or in $H^c(T)$ is omitted here, as they can be addressed in a similar manner.
    
    We proceed with the following two claims. 
    
    \noindent \textbf{Claim 1.} Assuming $x=0$ and $x$ is associated with a vertical domino in $V(T)$, then there must exist $j \leq d$ such that $v_j = 1$ and $v_1 = v_2 = \cdots = v_{j-1} = 0$. 
    \begin{proof}[Proof of Claim 1.]
        Figure \ref{fig:neeven} shows the case when $x$ is an entry of $M_e(T)$ (i.e., the image of an even point under the map $M_e$). Let $p_0,p_1,\dots,p_d$ be the even points corresponding to $x,v_1,\dots,v_{d}$, respectively. In the domino tiling $T$, there are four ways to cover the unit square marked $X$ in Figure \ref{fig:neeven}. For the two cases on the right, $\deg(p_1) = 4$, implying $v_1 = M_e(P_1) = 1$, and the claim holds. For the two cases on the left, $\deg(p_1)=3$, giving $v_1 = M_e(P_e) = 0$. We then proceed to examine the unit square between $p_2$ and $p_3$ in the same way. If $v_2 = 1$, the claim holds. If $v_2 = 0$, then we continue this process until we find $v_j = 1$ for some $j \leq d$. This process will stop: the worst case is when $v_1=\cdots=v_{d-1} = 0$, but this forces $v_d = 1$ since $p_d$ is on the boundary of the extended $\AD_n$ and $\deg(p_d)=4$.

        On the other hand, Figure \ref{fig:neodd} illustrates the other case when $x$ is an entry of $M_o(T)$ (i.e., the image of an odd point under the map $M_o$). Let $q_0,q_1,\dots,q_d$ be the odd points corresponding to $x,v_1,\dots,v_{d}$, respectively. In the domino tiling $T$, there are three possible ways to cover the unit square marked $X$ in Figure \ref{fig:neodd}. The leftmost case is not possible because $\deg(q_0)=2$ which implies $x = 1$ under the map $M_o$. For the two cases on the right, $\deg(q_1) < 4$, implying $v_1 \neq -1$. We then consider covering the unit square between $q_1$ and $q_2$ in a similar way. If this results in $v_1 = 1$, then the claim holds. If $v_1 = 0$, then this process continues until we find $v_j = 1$ for some $j \leq d$. This process will terminate for the same reason described in the previous paragraph: if $v_1= \cdots =v_{d-1} = 0$, then $v_d = 1$.
    \end{proof} 
    \begin{figure}[hbt!]
        \centering
        \subfigure[]{\label{fig:neeven}\includegraphics[height=0.16\textwidth]{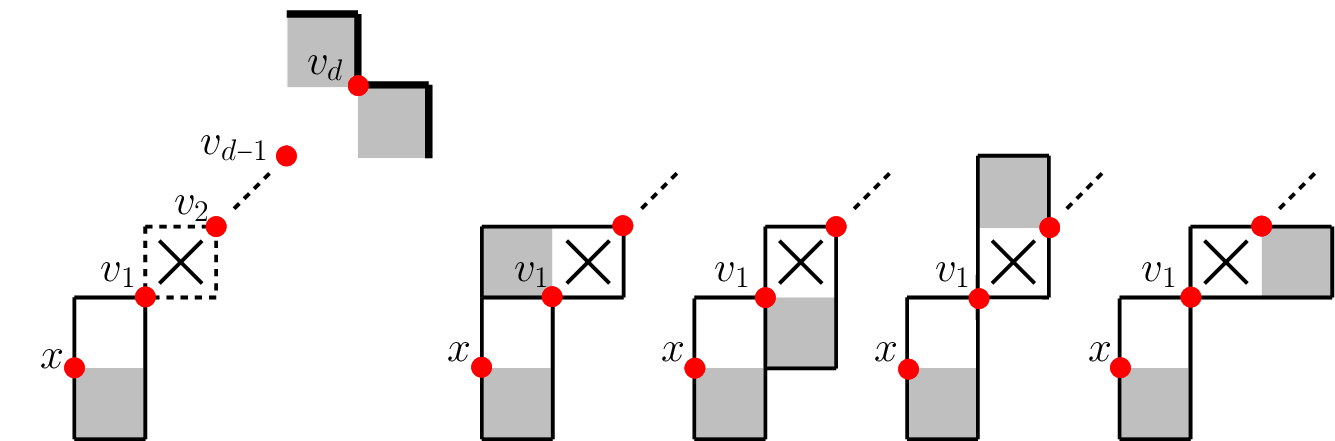}}
        \hspace{5mm}
        \subfigure[]{\label{fig:neodd}\includegraphics[height=0.16\textwidth]{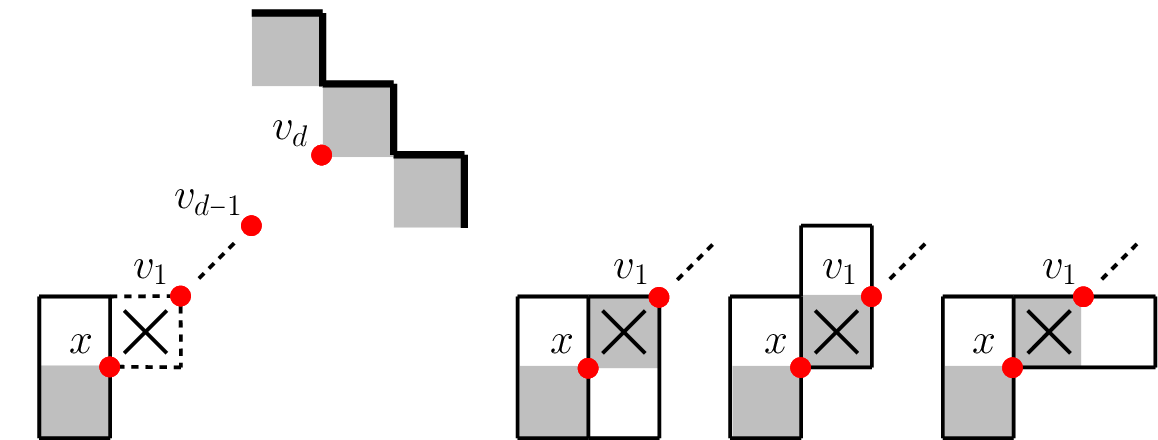}}
        \caption{The possible tilings around the vertical domino in $V(T)$ that $x$ is associated with in Claim 1, where (a) $x$ is an entry of $M_e(T)$, and (b) $x$ is an entry of $M_o(T)$.}\label{fig:nesum}
    \end{figure}
    \noindent \textbf{Claim 2.} We assume $x=0$ and $x$ is associated with a vertical domino in $V^c(T)$. If there exists $j \leq d$ such that $v_j \neq 0$ and $v_1 = v_2 = \cdots = v_{j-1} = 0$, then $v_j = -1$. 
    \begin{proof}[Proof of Claim 2.]
        We assume the contrary that $v_j = 1$. Figure \ref{fig:neeven2} shows the case when $x$ is an entry of $M_e(T)$. Let $p_0,p_1,\dots,p_j$ be the even points corresponding to $x,v_1,\dots,v_{j}$, respectively. We have $\deg(p_j) = 4$ and $\deg(p_i)=3$ for $i=1,\dots,j-1$. As a consequence, in the domino tiling $T$, there are two ways to cover the unit square marked $X$ in Figure \ref{fig:neeven2} and the edges of dominoes must contain the ``southwestern hook'' of each $p_i$ for $i=1,\dots,j-1$, that is, the union of the two edges pointing west and south from $p_i$. It turns out that there is only one way to cover the unit square between $x$ and $v_1$ which makes $\deg(p_0)=2$ and $x=M_e(p_0)=-1$, this contradicts to the assumption that $x=0$.
        
        Figure \ref{fig:neodd2} shows the other case when $x$ is an entry of $M_o(T)$. Let $q_0,q_1,\dots,q_j$ be the odd points corresponding to $x,v_1,\dots,v_{j}$, respectively. We have $\deg(q_j) = 2$ and $\deg(q_i)=3$ for $i=1,\dots,j-1$. Consequently, in the domino tiling $T$, there are two ways to cover the unit square marked $X$ in Figure \ref{fig:neodd2} and the edges of dominoes must contain the ``northeastern hook'' of each $q_i$ for $i=1,\dots,j-1$. This leads to $\deg(q_1) = 4$ and $v_1 = M_o(q_1) = -1$, which violates the assumption that $v_1 = 0$. 
    \end{proof}
    \begin{figure}[hbt!]
    \centering
    \subfigure[]{\label{fig:neeven2}\includegraphics[height=0.16\textwidth]{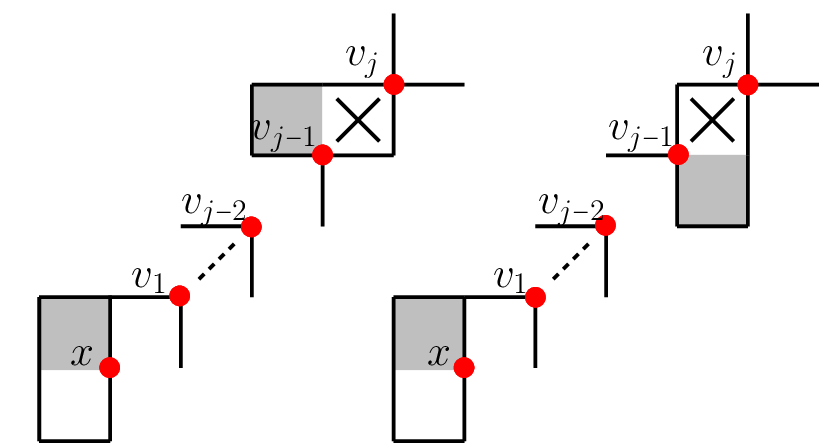}}
    \hspace{10mm}
    \subfigure[]{\label{fig:neodd2}\includegraphics[height=0.16\textwidth]{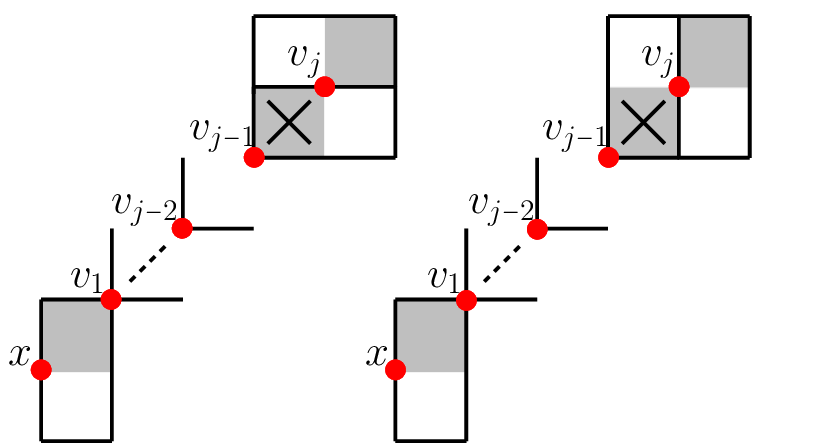}}
    \caption{The possible tilings around the vertical domino in $V^c(T)$ that $x$ is associated with in Claim 2, where (a) $x$ is an entry of $M_e(T)$, and (b) $x$ is an entry of $M_o(T)$.}\label{fig:nesum2}
    \end{figure}

    If $x=0$ and $x$ is associated with a vertical domino in $V(T)$, then Claim 1 implies that 
    \begin{equation*}
        \mathsf{ne}(x) = \sum_{i=1}^{j-1} v_j + v_j + \mathsf{ne}(v_j) = 0 +1 + 0 = 1.
    \end{equation*}
    Finally, when $x=0$ and $x$ is associated with a vertical domino in $V^c(T)$, if all the $v_i$'s are $0$, then $\mathsf{ne}(x) = 0$. If not, then Claim 2 implies that
    \begin{equation*}
        \mathsf{ne}(x) = \sum_{i=1}^{j-1} v_j + v_j + \mathsf{ne}(v_j) = 0 + (-1) + 1 = 0.
    \end{equation*}
    This completes the proof of Lemma \ref{lem.nesum}.
\end{proof}

Before stating the result of the second statistic $\tp$, we partition the set $V(T) = V_g(T) \cup V_b(T)$, where $V_g(T)$ consists of vertical dominoes lying on the ground (i.e., the bottom edge of a domino lies on $y=0$) and are not in $\AD_n$, and $V_b(T) = V(T) \setminus V_g(T)$. Notice that for a domino tiling $T$ of the extended $\AD_n$, $|V_g(T)|=n-1$. In the following lemma, the relation between statistics $\dinv_1$, $\dinv_2$ (Section \ref{sec:ADbijection}) and $\tp$ is presented.
\begin{lemma}\label{lem.dinv-tpsum}
    Given a domino tiling $T$ of the extended $\AD_n$, let $d_v$ be a vertical domino in the set $V_b(T)$. Suppose the entry $x$ of $M_e(T)$ and the entry $y$ of $M_o(T)$ are associated with $d_v$. Note that $x=0 \text{ or } -1$ under the map $M_e$, and $y=0 \text{ or } 1$ under the map $M_o$. Then
    \begin{equation}\label{eq.dinv-tp}
        \dinv_1(d_v) = \tp(x) - 1 \text{ and } \dinv_2(d_v) = \tp(y) - \delta_{y,0},
    \end{equation}
    where $\delta$ denotes the Kronecker delta function.
\end{lemma}
\begin{proof} 
    Since $d_v \in V_b(T)$, the four configurations of the domino pair with $d_v$ may exist in $T$. The situation when $d_v \in V_g(T)$ will be discussed in Remark \ref{rmk.dinv}. We consider the first case when $x$ is the entry of $M_e(T)$ and is on the line $\ell_2$; see Figure \ref{fig:tpeven} for an illustration. In this case, $x = -1$ or $x=0$, hence $\mathsf{ne}(x) = 1$. Viewing $d_v$ as the type I reference domino (Figure \ref{fig:dominotypeI}), we proceed with the following claim about the statistic $\mathsf{ne}$ and domino pairs. 
\begin{figure}[hbt!]
    \centering
    \subfigure[]{\label{fig:tpeven}\includegraphics[height=0.16\textwidth]{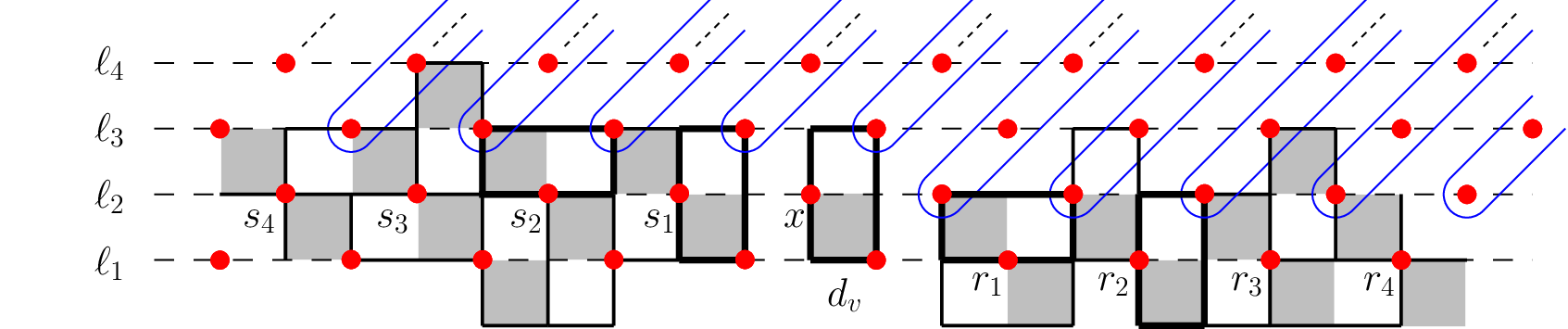}}
    \subfigure[]{\label{fig:tpodd}\includegraphics[height=0.16\textwidth]{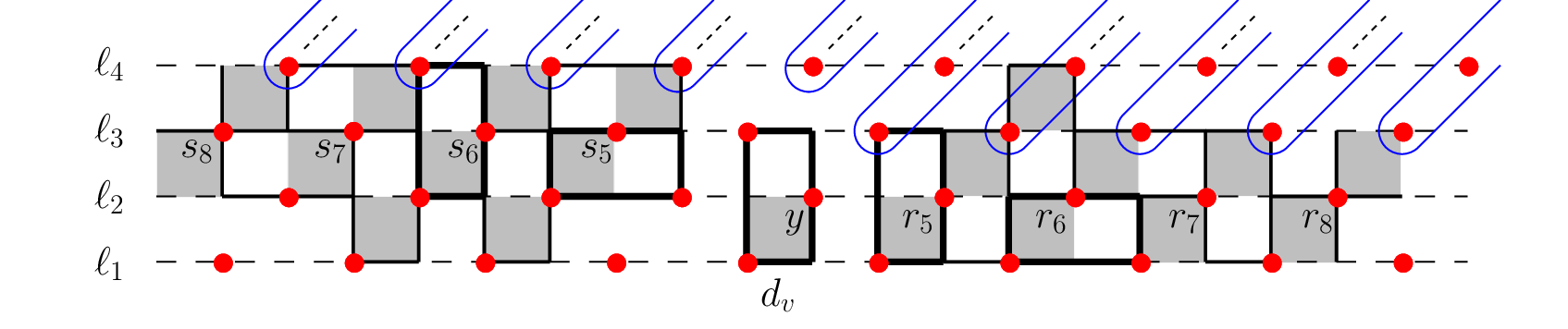}}
    \caption{An illustration of domino pairs and non-domino pairs involving $d_v$, dominoes forming a domino pair with $d_v$ are highlighted in bold, where (a) $x$ is associated with $d_v$ as an entry of $M_e(T)$, and (b) $y$ is associated with $d_v$ as an entry of $M_o(T)$.} \label{fig:tpsum}
\end{figure}

    \noindent \textbf{Claim 3.}  Let $u$ be an entry of $M_e(T)$ on the line $\ell_2$ and the left of $d_v$, or on the line $\ell_1$ and the right of $d_v$. Then the domino that $u$ is associated with forms a domino pair with $d_v$ if and only if $\mathsf{ne}(u) = 1$.
    \begin{proof}[Proof of Claim 3.]
        We discuss the following three possible values of $u$.
        \begin{itemize}
            \item $u=-1$. Since $u$ corresponds to an odd block in $T$ filled by two horizontal or two vertical dominoes, in either case, one of the dominoes forms a domino pair with $d_v$ (such as the dominoes that $s_1$ and $r_1$ are associated with in Figure \ref{fig:tpeven}), contributing $1$ to the $\dinv_1$ statistic. Lemma \ref{lem.nesum} implies $\mathsf{ne}(u)=1$.
            \item $u=0$. If $u$ is associated with a domino which forms a domino pair with $d_v$ (such as the dominoes that $s_2$ and $r_2$ are associated with in Figure \ref{fig:tpeven}), that is, they belong to the set $V(T) \cup H(T)$, then Lemma \ref{lem.nesum} implies $\mathsf{ne}(u)=1$. On the other hand, if $u$ is associated with a domino which does not form a domino pair with $d_v$ (for example, the dominoes that $s_3$ and $r_3$ are associated with in Figure \ref{fig:tpeven}), that is, they belong to the set $V^c(T) \cup H^c(T)$, then $\mathsf{ne}(u)=0$ by Lemma \ref{lem.nesum}.
            \item $u=1$. The even point corresponding to $u$ has degree $4$, drawn as $s_4$ and $r_4$ in Figure \ref{fig:tpeven}. No matter how we place dominoes incident to this even point, none of them forms a domino pair with $d_v$. By Lemma \ref{lem.nesum}, $\mathsf{ne}(u)=0$. 
        \end{itemize}
        From the above discussion, we see that $u$ is associated with a domino that forms a domino pair with $d_v$ if and only if $\mathsf{ne}(u)=1$.    
    \end{proof}
    In this case, we observe that the top partial sum of $x$ can be written as
    \begin{equation*}
        \tp(x) = \mathsf{ne}(x)+ \sum_{u} \mathsf{ne}(u),
    \end{equation*}
    where $u$ runs over all the entries of $M_e(T)$ on $\ell_2$ and the left of $d_v$ or on $\ell_1$ and the right of $d_v$. By Claim 3, $\dinv_1(d_v) = \tp(x) - \mathsf{ne}(x) = \tp(x)-1$.

    Now, we look at the second case when $y$ is the entry of $M_o(T)$ and is on the line $\ell_2$; see Figure \ref{fig:tpodd} for an illustration. In this case, $y = 1$ or $y=0$, hence $\mathsf{ne}(y) = \delta_{y,0}$. Viewing $d_v$ as the type II reference domino (Figure \ref{fig:dominotypeII}), we proceed similarly with the following claim.

    \noindent \textbf{Claim 4.}  Let $u$ be an entry of $M_o(T)$ on the line $\ell_2$ and the right of $d_v$, or on the line $\ell_3$ and the left of $d_v$. Then the domino that $u$ is associated with forms a domino pair with $d_v$ if and only if $u+\mathsf{ne}(u) = 1$.
    \begin{proof}[Proof of Claim 4.]
        We similarly analyze the three possible values of $u$, as follows:
        \begin{itemize}
            \item $u=1$. Since $u$ corresponds to an even block in $T$ filled by two horizontal or two vertical dominoes, in either case, one of the dominoes forms a domino pair with $d_v$ (such as the dominoes that $s_5$ and $r_5$ are associated with in Figure \ref{fig:tpodd}), contributing $1$ to the $\dinv_2$ statistic. Lemma \ref{lem.nesum} implies $u + \mathsf{ne}(u)=1$.
            \item $u=0$. If $u$ is associated with a domino which forms a domino pair with $d_v$ (such as the dominoes that $s_6$ and $r_6$ are associated with in Figure \ref{fig:tpodd}), that is, they belong to the set $V(T) \cup H(T)$, then Lemma \ref{lem.nesum} implies $u+\mathsf{ne}(u)=1$. On the other hand, if $u$ is associated with a domino which does not form a domino pair with $d_v$ (for example, the dominoes that $s_7$ and $r_7$ are associated with in Figure \ref{fig:tpodd}), that is, they belong to the set $V^c(T) \cup H^c(T)$, then $u+\mathsf{ne}(u)=0$ by Lemma \ref{lem.nesum}.
            \item $u=-1$. The odd point corresponding to $u$ has degree $4$, shown as $s_8$ and $r_8$ in Figure \ref{fig:tpodd}. No matter how we place dominoes incident to this odd point, none of them forms a domino pair with $d_v$. By Lemma \ref{lem.nesum}, $u+\mathsf{ne}(u)=0$. 
        \end{itemize}
        From the above discussion, Claim 4 holds.   
    \end{proof}
    In this case, $\tp(y)$ is slightly different from $\tp(x)$ mentioned in the first case, it is given by 
    \begin{equation*}
        \tp(y) = \mathsf{ne}(y) + \sum_{u} (u+\mathsf{ne}(u)),
    \end{equation*}
    where $u$ runs over all the entries of $M_o(T)$ on $\ell_3$ and the left of $d_v$ or on $\ell_2$ and the right of $d_v$. By Claim 4, $\dinv_2(d_v) = \tp(y) - \mathsf{ne}(y) = \tp(y)- \delta_{y,0}$. This completes the proof of Lemma \ref{lem.dinv-tpsum}.
\end{proof}
\begin{remark}\label{rmk.dinv}
    Let $T,x,y$ be the same as in Lemma \ref{lem.dinv-tpsum}. The identity $\dinv_1(d_v) = \tp(x) -1$ does not hold for $d_v \in V_g(T)$. The issue arises when viewing $d_v$ as the type I reference domino, one of the configurations of the domino pairs involving $d_v$ does not exist in $T$ (see the vertical domino that $r_2$ is associated with in Figure \ref{fig:tpeven}), and this discrepancy is not accurately reflected in $M_e(T)$. Hence, $\dinv_1(d_v) < \tp(x) -1$. However, if we attach $2n$ vertical dominoes at the bottom of $T$ (like the dominoes highlighted in blue in Figure \ref{fig:exTodd}), then this does give the correct identity $\dinv_1(d_v) = \tp(x) -1$ for $d_v \in V_g(T)$. On the other hand, all four configurations of the domino pairs could exist when viewing $d_v$ as the type II reference domino. Thus, $\dinv_2(d_v) = \tp(y) - \delta_{y,0}$ remains valid for $d_v \in V_g(T)$. 
\end{remark}

We are now ready to prove Proposition \ref{prop.dinv-shuffle}.
\begin{proof}[Proof of Proposition \ref{prop.dinv-shuffle}.]  
    We compute the dinv statistic on the odd-deficient tiling $T_o$ using the type I reference domino and on the even-deficient tiling $T_e$ using the type II reference domino. We also partition the set $V(\hat{T}_o)$ and $V(\hat{T}_e)$ described in the paragraph before Lemma \ref{lem.dinv-tpsum}. By \eqref{eq.def-domino-dinv-reduce}, \eqref{eq.dominopairs}, and \eqref{eq.def-dinv-odd/even}, we may express the dinv statistic as
    \begin{align}
        \dinv(T_o) = \dinv(\hat{T}_o) & = \binom{n}{2} + \sum_{d_v \in V_g(\hat{T}_o)} \dinv_1(d_v) + \sum_{d_v \in V_b(\hat{T}_o)} \dinv_1(d_v), \label{eq.dinvTo1}\\
        \dinv(T_e) = \dinv(\hat{T}_e) & = \binom{n+1}{2} + \sum_{d^{\prime}_v \in V_g(\hat{T}_e)} \dinv_2(d^{\prime}_v) + \sum_{d^{\prime}_v \in V_b(\hat{T}_e)} \dinv_2(d^{\prime}_v). \label{eq.dinvTe1}
    \end{align}  
    Since $\hat{T}_e$ is the domino tiling of the extended $\AD_{n+1}$, $|V_g(\hat{T}_e)| = n$. For each vertical domino $d^{\prime}_v \in V_g(\hat{T}_e)$, it is not hard to see that there are $n$ dominoes that form a domino pair with $d^{\prime}_v$ (viewing as the type II reference domino), thus $\dinv_2(d^{\prime}_v) = n$. Consequently,
    \begin{equation}\label{eq.dinvVg}
        \sum_{d^{\prime}_v \in V_g(\hat{T}_e)} \dinv_2(d^{\prime}_v) = n^2.
    \end{equation}
    Comparing \eqref{eq.dinv-shuffle} with \eqref{eq.dinvTo1}, \eqref{eq.dinvTe1}, and \eqref{eq.dinvVg}, our proof will be complete once we prove the following identity.
    \begin{equation}\label{eq.dinv-shuffle-temp}
        \sum_{d^{\prime}_v \in V_b(\hat{T}_e)} \dinv_2(d^{\prime}_v) = \dinv(\hat{T}_o).
    \end{equation}

    The vertical dominoes in the set $V_b(\hat{T}_e)$ come from the vertical dominoes in the set $V(\hat{T}_o) = V_g(\hat{T}_o) \cup V_b(\hat{T}_o)$ under the domino shuffling by moving each such domino one unit to its left. To prove \eqref{eq.dinv-shuffle-temp}, we analyze these vertical dominoes in the following two cases. 
    
    \noindent Case (1). A vertical domino $d^{\prime}_v$ in $V_b(\hat{T}_e)$ is obtained from a vertical domino $d_v$ in $V_b(\hat{T}_o)$. We assume that an entry $x$ of $M_e(T_o)$ is associated with $d_v$, it is easy to see that after shuffling, $d^{\prime}_v$ is the domino that the same entry $x$ is associated with, where $x$ is an entry of $M_o(T_e)$, the same extended ASM as $M_e(T_o)$. Based on the definition of $\hat{T}_o$ and $\hat{T}_e$, we note that $d_v$ is not contained in odd blocks of $\hat{T}_o$ and $d^{\prime}_v$ is not contained in even blocks of $\hat{T}_o$. This fact implies that $x=0$. By Lemma \ref{lem.dinv-tpsum}, for all $d_v \in V_b(\hat{T}_o)$, 
    \begin{equation}\label{eq.dinv-shufflecase1}
        \dinv_1(d_v) = \tp(x) -1 = \dinv_2(d^{\prime}_v), 
    \end{equation}
    where $d^{\prime}_v$ is the shuffle of $d_v$. 
    
    \noindent Case (2). A vertical domino $d^{\prime}_v$ in $V_b(\hat{T}_e)$ is obtained from a vertical domino $d_v$ in $V_g(\hat{T}_o)$. As explained in Remark~\ref{rmk.dinv}, Lemma~\ref{lem.dinv-tpsum} does not hold unless $2n$ vertical dominoes are attached at the bottom of $\hat{T}_o$. In fact, these attached dominoes increase the number of domino pairs involving $d_v \in V_g(\hat{T}_o)$. To be more precise, the $\dinv_1$ statistic of each $d_v$ from left to right is increased by $n-1,n-2,\dotsc,1$. As a result, 
    \begin{equation}\label{eq.dinv-shufflecase2}
        \sum_{d_v \in V_g(\hat{T}_o)} \dinv_1(d_v) + \binom{n}{2} = \sum_{x \in M_e(\hat{T}_o)} (\tp(x) - 1) = \sum_{d^{\prime}_v}\dinv_2(d^{\prime}_v),
    \end{equation}
    where $x$ is associated with each domino $d_v \in V_g(\hat{T}_o)$ and $d^{\prime}_v$ is the shuffle of each $d_v$. Note that the second equality follows directly from Lemma \ref{lem.dinv-tpsum} and the fact that $x=0$.

    Summing over all $d_v \in V_b(\hat{T}_o)$ on both sides of \eqref{eq.dinv-shufflecase1} and combining with \eqref{eq.dinv-shufflecase2}, we obtain
    \begin{equation}\label{eq.dinv-shufflecombine}
        \sum_{d_v \in V_b(\hat{T}_o)} \dinv_1(d_v) + \sum_{d_v \in V_g(\hat{T}_o)} \dinv_1(d_v) + \binom{n}{2} = \sum_{d^{\prime}_v \in V_b(\hat{T}_e)} \dinv_2(d^{\prime}_v).
    \end{equation}
    The left-hand side of \eqref{eq.dinv-shufflecombine} coincides with $\dinv(\hat{T}_o)$, we then obtain the desired identity \eqref{eq.dinv-shuffle-temp}. This completes the proof of Proposition \ref{prop.dinv-shuffle}.
\end{proof}

\subsection{Proof of Theorem \ref{thm:introAD}}\label{sec:pfADformula}

We remind the reader from \eqref{eq.P-polynomial} and \eqref{eq.ADdef} that the polynomial $\AD_n(z;q,t)$ is the generating polynomial of domino tilings of $R_{\lambda}$, where $\lambda = (n^n)$.
\begin{equation}\label{eq.ADpolynomial}
    \AD_n(z;q,t) = \sum_{T \in \T \left( R_{\lambda} \right)} z^{\diags(T)} q^{\area(T)} t^{\dinv(T)}.
\end{equation}
For a subset $\mathcal{D}$ of $\T(R_{\lambda})$, we define the generating polynomial of domino tilings of $\mathcal{D}$ by
\begin{equation}
    \AD_n(\mathcal{D};z;q,t) = \sum_{T \in \mathcal{D}}z^{\diags(T)} q^{\area(T)} t^{\dinv(T)}.
\end{equation}

We first discuss the impact of changing the orientation of dominoes—--either from two horizontal to two vertical dominoes or vice versa—--in each odd or even block on our three statistics $\diags$, $\area$, and $\dinv$. In each odd block of $\hat{T}_o$, we observe that changing two horizontal dominoes to two vertical dominoes decreases $\area(\hat{T}_o)$ by $1$ and $\diags(\hat{T}_o)$ by $1$. Thus, changing the orientation of dominoes in each odd block of $\hat{T}_o$ preserves $\area(\hat{T}_o) - \diags(\hat{T}_o)$. For a vertical domino $d_v \in V(\hat{T_o})$, it is not hard to see that changing two horizontal dominoes to two vertical dominoes in each odd block of $\hat{T}_o$ keeps the number of domino pairs involving $d_v$ unchanged (viewing $d_v$ as the type I reference domino). This implies that changing the orientation of dominoes in each odd block preserves $\dinv_1(d_v)$.

Let $\widetilde{T}_o$ be the set consisting of domino tilings obtained from $T_o$ by filling the odd blocks of $T_o$ with either two horizontal or two vertical dominoes. Note that $\hat{T}_o$ is one of the elements of $\widetilde{T}_o$ and $|\widetilde{T}_o|=2^m$. The total weight of $\widetilde{T}_o$ can be obtained as follows. Starting with the odd-deficient tiling $T_o$, we sequentially fill the $i$th odd block with dominoes for $i=1,2,\dots,m$ (the order of filling does not matter). If the $i$th odd block is filled by two horizontal dominoes, $\diags(T_o)$ increases by $1$ while $\area(T_o)$ and $\dinv(T_o)$ remain unchanged (due to the way we define them \eqref{eq.def-area-To} and \eqref{eq.def-dinv-odd/even}). This contributes a factor of $z$ to the weight. On the other hand, if the $i$th odd block is filled by two vertical dominoes, $\diags(T_o)$ remains the same, $\area(T_o)$ decreases by $1$ and we assume $\dinv(T_o)$ increases by $k_i=\dinv_1(d_i)$, where $d_i$ is one of the vertical dominoes filled, with its bottom unit square colored black. This follows from the discussion in the previous paragraph. This adds a factor of $q^{-1}t^{k_i}$ to the weight. 

In total, filling the $i$th odd block contributes a factor of $z+q^{-1}t^{k_i}$ to the weight. Therefore, the total weight of $\widetilde{T}_o$ is given by
\begin{equation}\label{eq.T_oweight}
    \AD_n(\widetilde{T}_o;z;q,t) = z^{\diags(T_o)} q^{\area(T_o)} t^{\dinv(T_o)} \prod_{i=1}^{m}(z+q^{-1}t^{k_i}).
\end{equation}
Summing over all odd-deficient tilings in \eqref{eq.T_oweight} gives the generating polynomial of domino tilings of the extended $\AD_n$.
\begin{equation}\label{eq.ADsumodd}
    \AD_n(z;q,t) = \sum_{T_o}\AD_n(\widetilde{T}_o;z;q,t).
\end{equation}

In each even block of $\hat{T}_e$, we observe that changing two horizontal dominoes to two vertical dominoes decreases $\diags(\hat{T}_e)$ by $1$ but $\area(\hat{T}_e)$ remains unchanged. For a vertical domino $d^{\prime}_v \in V(\hat{T_e})$, it is not hard to see that changing the orientation of dominoes in each even block of $\hat{T}_e$ keeps the number of domino pairs involving $d^{\prime}_v$ unchanged (viewing $d^{\prime}_v$ as the type II reference domino). Thus, $\dinv_2(d^{\prime}_v)$ is preserved. 

Let $\widetilde{T}_e$ be the set consisting of domino tilings obtained from $T_e = S(T_o)$ by filling the even blocks of $T_e$ with either two horizontal or two vertical dominoes. Note that $\hat{T}_e$ is one of the elements of $\widetilde{T}_e$ and $|\widetilde{T}_e|=2^{m+n+1}$. The total weight of $\widetilde{T}_e$ can be obtained similarly. Starting with the even-deficient tiling $T_e$, we sequentially fill the $i$th even block with dominoes for $i=1,2,\dots,m+n+1$ (the order of filling does not matter). If the $i$th even block is filled by two horizontal dominoes, $\diags(T_e)$ increases by $1$ while $\area(T_e)$ and $\dinv(T_e)$ remain unchanged (due to \eqref{eq.def-area-Te} and \eqref{eq.def-dinv-odd/even}). This contributes a factor of $z$ to the weight. On the other hand, if the $i$th even block is filled by two vertical dominoes, $\diags(T_e)$ and $\area(T_e)$ remain the same, and we assume $\dinv(T_e)$ increases by $\ell_i=\dinv_2(d^{\prime}_i)$, where $d^{\prime}_i$ is one of the vertical dominoes filled, with its bottom unit square colored black. This follows from the discussion in the previous paragraph and adds a factor of $t^{\ell_i}$ to the weight. 

In total, filling the $i$th even block contributes a factor of $z+t^{\ell_i}$ to the weight. Therefore, the total weight of $\widetilde{T}_e$ is given by
\begin{equation}\label{eq.T_eweight}
    \AD_{n+1}(\widetilde{T}_e;z;q,t) = z^{\diags(T_e)} q^{\area(T_e)} t^{\dinv(T_e)} \prod_{i=1}^{m+n+1}(z+t^{\ell_i}).
\end{equation}
Summing over all even-deficient tilings in \eqref{eq.T_eweight} gives the generating polynomial of domino tilings of the extended $\AD_{n+1}$.
\begin{equation}\label{eq.ADsumeven}
    \AD_{n+1}(z;q,t) = \sum_{T_e}\AD_{n+1}(\widetilde{T}_e;z;q,t).
\end{equation}

In the following lemma, we show how the dinv statistic of the vertical dominoes contained in each odd or even block changes under the domino shuffling. This result will be used later to prove Theorem \ref{thm:introAD}. 
\begin{lemma}\label{lem.dinv-odd/even}
    Let $T$ be a domino tiling of the extended $\AD_n$ with $m$ odd blocks. Suppose the odd blocks of $T_o$ and the even blocks of $T_e=S(T_o)$ are filled with two vertical dominoes. Let $d_i$ (resp., $d^{\prime}_j$) be the vertical domino with its bottom unit square colored black contained in the $i$th odd block (resp., the $j$th even block), for each $i=1,\dots,m$ (resp., $j=1,\dots,m+n+1$).
    Assuming $\dinv_1(d_i) = k_i$ and $\dinv_2(d^{\prime}_j) = \ell_j$, then as multisets,
    \begin{equation}\label{eq.dinv-oddevenblock}
        \{\ell_1,\ell_2,\dotsc,\ell_{m+n+1} \} \setminus \{k_1,k_2,\dotsc,k_m \} = \{0,1,\dotsc,n\}.
    \end{equation}
\end{lemma}
\begin{proof}
    Let $M_e(T)$ be the extended ASM of $T$. Define the \emph{reading word} of $T$ by recording the non-zero entries of $M_e(T)$ from top to bottom and right to left. We write $w(T) = w_1w_2 \cdots w_p$ with $w_i \in\{-1,1\}$ for the reading word of $T$ and set $\tp(w_i) = \sum_{s=1}^{i-1}w_s$ to be the top partial sum (Section \ref{sec:shuffle-dinv}) of $w_i$ in $M_e(T)$. By Lemma \ref{lem.dinv-tpsum}, 
    \begin{align*}
        k_i &= \dinv_1(d_i) = \tp(w_{a_i})-1 = \sum_{s=1}^{a_i-1}w_s - 1, \\
        \ell_j &= \dinv_2(d^{\prime}_j) = \tp(w_{b_j}) = \sum_{s=1}^{b_j-1}w_s,
    \end{align*}
    where $w_{a_i} = -1$ and $w_{b_j} = 1$ for some positive integers $a_i$ and $b_j$. 
    
    Next, we find $(1,-1)$ pairs from the reading word $w(T)$ as follows. From left to right, for each $-1$ in $w(T)$, find the nearest unpaired $1$ on its left, and form a $(1,-1)$ pair. For example, the reading word of the extended ASM given in Figure \ref{fig:extasm} is $w = 1,1,-1,1,1,1$ and $(w_2,w_3)$ forms the only $(1,-1)$ pair. Since the number of $1$'s minus the number of $-1$'s of an ASM of order $n+1$ equals $n+1$, it remains $n+1$ unpaired $1$'s in the reading word $w(T)$ for all domino tilings $T$ of the extended $\AD_n$.

    Since we assume $T$ has $m$ odd blocks, the reading word $w(T)$ contains $m$ $-1$'s, resulting in $m$ $(1,-1)$ pairs. If $w_{b_0} = 1$ is paired with $w_{a_0} = -1$ $(b_0<a_0)$ in the reading word $w(T)$, then we have the same number of $1$'s and $-1$'s in between $w_{b_0}$ and $w_{a_0}$. This implies that 
    \begin{equation}
        \tp(w_{a_0})-1 = \sum_{s=1}^{a_0-1}w_s - w_{b_0} = \sum_{s=1}^{b_0-1}w_s = \tp(w_{b_0}).
    \end{equation}
    Thus, the associated vertical dominoes contribute the same number of domino pairs, that is, $k_{a_0} = \dinv_1(d_{a_0}) = \dinv_2(d^{\prime}_{b_0}) = \ell_{b_0}$. This implies that these dinv contributions cancel out on the left-hand side of \eqref{eq.dinv-oddevenblock}, it remains to show that the dinv contribution from those unpaired $1$'s in the reading word $w(T)$ is given by $\{0,1,\dots,n\}$. 

    For each unpaired $w_{b_j}=1$ in the reading word $w(T)$, it is easy to see that  
    \begin{equation}
        \dinv_2(d^{\prime}_{b_j}) = \tp(w_{b_j}) = \sum_{s=1}^{b_j-1} w_s = j-1,
    \end{equation}
    for $j=1,2,\dotsc,n+1$. In other words, the left-hand side of \eqref{eq.dinv-oddevenblock} is given by $\{0,1,\dotsc,n\}$, as desired.
\end{proof}

Recall that Theorem \ref{thm:introAD} states that the polynomial $\AD_n(z;q,t)$ is given by the product formula. We close this section with the proof of Theorem \ref{thm:introAD}.
\begin{proof}[Proof of Theorem \ref{thm:introAD}]
    We prove it by induction on $n$. For the base case $n=1$, the extended $\AD_1$ is a $2 \times 2$ square which has two domino tilings: the tiling with two horizontal dominoes is weighted by $z$ while the tiling with two vertical dominoes is weighted by $1$. Thus, $\AD_1(z;q,t)=z+1$ matches \eqref{eq.ADproduct} when $n=1$. 
    
    Now, we assume the statement holds for $n$. Substituting $zq^{-1}$ for $z$ in \eqref{eq.T_oweight} yields
    \begin{equation}\label{eq.proofADsub}
        \AD_n(\widetilde{T}_o;zq^{-1};q,t) = z^{\diags(T_o)} q^{\area(T_o)-\diags(T_o)} t^{\dinv(T_o)} q^{-m} \prod_{i=1}^{m}(z+t^{k_i}).
    \end{equation}
    By Lemma \ref{lem.dinv-odd/even}, we can relate the linear products in \eqref{eq.T_eweight} and \eqref{eq.proofADsub} as follows
    \begin{equation}
        \prod_{i=1}^{m}(z+t^{\ell_i}) = \prod_{i=1}^{m}(z+t^{k_i}) \prod_{i=0}^{n}(z+t^i).
    \end{equation}
    In \eqref{eq.T_eweight}, we replace $\area(T_e)$ with $\area(T_o)$ by Proposition \ref{prop:area-shuffle} and replace $\dinv(T_e)$ with $\dinv(T_o)$ by Proposition \ref{prop.dinv-shuffle}. It is clear that $\diags(T_e) = \diags(T_o)$. Combining the above discussion, we can rewrite \eqref{eq.T_eweight} as 
    \begin{align}
        \AD_{n+1}(\widetilde{T}_e;z;q,t) & = z^{\diags(T_o)} q^{\area(T_o)-\dinv(T_o)} t^{\dinv(T_o)} \prod_{i=1}^{m}(z+t^{k_i}) \nonumber\\
        & \quad \times q^{n(2n+1)-m}t^{n^2+\binom{n+1}{2}} \prod_{i=0}^{n}(z+t^i) \nonumber\\
        & = \AD_n(\widetilde{T}_o;zq^{-1};q,t) q^{n(2n+1)}t^{n^2+\binom{n+1}{2}} \prod_{i=0}^{n}(z+t^i). \label{eq.proofADalign}
    \end{align}
    
    Since $T_e$ is the image of $T_o$ under the domino shuffling, we then sum over all odd-deficient tilings $T_o$ on both sides of \eqref{eq.proofADalign}. By \eqref{eq.ADsumeven}, we obtain $\AD_{n+1}(z;q,t)$ on the left-hand side. On the right-hand side, by \eqref{eq.ADsumodd} and the induction hypothesis, we have
    \begin{align*}
       & \quad \AD_n(zq^{-1};q,t) q^{n(2n+1)}t^{n^2+\binom{n+1}{2}} \prod_{i=0}^{n}(z+t^i) \nonumber \\
        & = \left( (qt)^{n^2(n-1)/2}\prod_{i,j \geq 0 \text{ and } i+j<n}(zq^{-1}+q^it^j) \right) q^{n(n+1)/2} q^{n(3n+1)/2} t^{n(3n+1)/2} \prod_{i=0}^{n}(z+t^i) \nonumber \\
        & = \left( (qt)^{n^2(n-1)/2}\prod_{i \geq 1, j \geq 0 \text{ and } i+j<n+1}(z+q^it^j) \right) (qt)^{n(3n+1)/2} \prod_{i=0}^{n}(z+t^i) \nonumber \\ 
        & = (qt)^{(n+1)^2 n/2}\prod_{i,j \geq 0 \text{ and } i+j<n+1}(z+q^it^j).
    \end{align*}
This completes the proof of \cref{thm:introAD}.
\end{proof}

\section{An Open Problem}\label{sec:open}

In this section, we formulate a conjecture regarding the polynomials $P_\lambda(z;q,t)$ (see \eqref{eq.P-polynomial}).

In Theorem \ref{thm:introAD}, we prove that the polynomial $P_\lambda(z;q,t)$ factors nicely when $\lambda = (n^n)$. Based on computations for small partitions, $P_\lambda(z;q,t)$ appear to have many factors of the form $(z+q^it^j)$. We conjecture that these linear factors are determined by the \emph{rank} of the partition $\lambda$, that is, the size of the Durfee square contained in $\lambda$ (\cite[Page 289]{ECII}). We have verified it for all partitions of size at most $9$.
\begin{conjecture}\label{conj:divisIntro}
    For any partition $\lambda$ of rank $r$, then the polynoimal $P_{\lambda}(z;q,t)$ is divisible by linear factors $(z+q^it^j)$ for $i,j \geq 0$ and $i+j<r$, 
    \begin{equation}
        P_{\lambda}(z;q,t) = Q_{\lambda}(z;q,t) \cdot \prod_{i,j\geq 0 \text{ and } i+j<r}(z+q^it^j),
    \end{equation}
    where the quotient $Q_{\lambda}(z;q,t)$ is a polynomial in $\mathbb{Z}_{\geq 0}[z,q,t]$.
\end{conjecture}

Theorem \ref{thm:introAD} directly confirms that Conjecture \ref{conj:divisIntro} holds for $\lambda=(n^n)$, where $\lambda$ is equal to its Durfee square. In general, the rank of $\lambda$ is equal to the number of nonzero border strip lengths $n_0,n_1,\dots,n_k$, and therefore the number of nontrivial paths in a $\lambda$-family of Schr\"oder paths. It would be an interesting problem to prove Conjecture~\ref{conj:divisIntro} either through the combinatorics of domino shuffling or using symmetric functions and Macdonald polynomials via Theorem \ref{thm:MacdonaldIntro}.

\section*{Acknowledgements}

This work was initiated at the Early-Career Conference in Combinatorics in June 2024. The second author thanks the conference organizers and the University of Illinois Urbana-Champaign for their hospitality. The authors thank James Haglund for pointing out the reference \cite[Chapter 6]{H08} and Nicholas Loehr for helpful conversations, especially Theorem \ref{thm:qtsym}.

\begingroup
\setstretch{1}
\bibliographystyle{plain}
\bibliography{refs.bib} 
\endgroup

\end{document}